\documentclass[11pt,reqno]{amsart}

\usepackage[usenames,svgnames,dvipsnames]{xcolor}
\definecolor{verde}{rgb}{0,0.5,0}

\definecolor{laranja}{rgb}{0.95,0.45,0}

\definecolor{vermelho}{rgb}{0.666,0,0}

\usepackage{latexsym}

\usepackage{amsthm,amsmath,amsfonts,amssymb}   

\usepackage{enumitem}
\setlist[enumerate]{label={\bf(\alph*)},ref={(\alph*)},align=left,leftmargin=*,itemindent=0cm,labelindent=\parindent,labelsep=0.5ex}   

\usepackage[pdftex]{graphicx}

\usepackage[all]{xy}
\usepackage{tikz-cd}

\newtheorem{theorem}{Theorem}[section]
\newtheorem{proposition}[theorem]{Proposition}
\newtheorem{corollary}[theorem]{Corollary}
\newtheorem{lemma}[theorem]{Lemma}

\theoremstyle{definition}
\newtheorem{definition}[theorem]{Definition}
\newtheorem{example}[theorem]{Example}

\theoremstyle{remark}
\newtheorem{remark}[theorem]{Remark}

\numberwithin{equation}{section}

\newcommand{\hol}{\text{hol}}
\newcommand{\I}{\mathcal{I}}
\newcommand{\J}{\mathcal{J}}
\newcommand{\Ker}{\ensuremath{\mathrm{\hspace{0.25ex}Ker\hspace{0.25ex}}}}
\newcommand{\proj}{\text{proj}}
\newcommand{\R}{\mathbb{R}}
\newcommand{\rank}{\ensuremath{\mathrm{\hspace{0.25ex}rank\hspace{0.25ex}}}}
\newcommand{\scal}[2]{\langle #1, #2 \rangle}
\newcommand{\Scal}[2]{\left\langle #1, #2 \right\rangle}
\newcommand{\sym}{\ensuremath{\mathrm{\hspace{0.25ex}sym\hspace{0.25ex}}}}
\newcommand{\tr}{\ensuremath{\mathrm{\hspace{0.25ex}tr\hspace{0.25ex}}}}

\renewcommand{\dim}{\ensuremath{\mathrm{\hspace{0.25ex}dim\hspace{0.25ex}}}}
\renewcommand{\L}{\mathcal{L}}

\def\RR{\mathbb{R}}

\newcommand{\F}{\ensuremath{\mathcal{F}}}

\title[Finsler submersions and attainable sets]{Traveling along horizontal broken geodesics of a homogenous Finsler submersion}

\author[Alexandrino]{Marcos M. Alexandrino}
\author[Escobosa]{Fernando M. Escobosa}
\author[Inagaki]{Marcelo K. Inagaki}

\address[Alexandrino,Inagaki, Escobosa]{Universidade de S\~ao Paulo, Instituto de Matem\'atica e Estat\'istica, Rua do Mat\~{a}o 1010,05508 090 S\~ao Paulo, Brazil}
\email{marcosmalex@yahoo.de, m.alexandrino@usp.br, kodi.inagaki@gmail.com, fme.mat@gmail.com}

\thanks{Marcos M. Alexandrino was supported by grant $\#$2016/23746-6, S\~{a}o Paulo Research Foundation (FAPESP).
The second author 
was  partially supported by the Coordena\c{c}\~ao de Aperfei\c{c}oamento de Pessoal de N\'ivel Superior -- Brasil (CAPES) -- Finance Code 001.
}

\date{\today}

\subjclass[2000]{}

\begin{document}

\begin{abstract}

In this paper, we discuss how to travel along horizontal broken geodesics of a homogenous Finsler submersion, 
i.e., we study, what in Riemannian geometry was called by Wilking, the dual leaves. 
More precisely, we investigate the attainable sets $\mathcal{A}_{q}(\mathcal{C})$ of the set of analytic vector fields $\mathcal{C}$ 
determined by the family of  horizontal unit geodesic vector fields $\mathcal{C}$ 
to the fibers $\F=\{\rho^{-1}(c)\}$ of a homogenous analytic Finsler submersion $\rho: M\to B$.
 Since reverse of geodesics don't need to be 
geodesics in Finsler geometry,
 one  can have examples on non compact Finsler manifolds $M$ where the attainable sets 
(the dual leaves) are no longer  orbits or even   submanifolds. 
Nevertheless we prove that, when $M$ is compact and the orbits of $\mathcal{C}$ are embedded, 
then the attainable sets coincide with the orbits. Furthermore,
if the flag curvature  is  positive 
then $M$ coincides with the attainable set of each point.
In other words, fixed two points of $M$, one can travel from one point  to the other 
along horizontal broken geodesics. 

In addition, we show that  each orbit $\mathcal{O}(q)$ of $\mathcal{C}$ associated 
to a  singular Finsler foliation  coincides with
$M$, when the flag curvature is positive, i.e, we prove Wilking's result in Finsler context. In particular
we review Wilking's transversal Jacobi fields in Finsler case. 

\end{abstract}

\maketitle

\section{Introduction}

Given a Riemannian submersion $\rho:M\to B$ and the Riemannian foliation  $\F=\{ \rho^{-1}(c) \}_{c\in B}$,  we 
can associate to it the so called \emph{dual foliation} $\F^{\#}=\{L^{\#}_{q} \}$, where  
 each leaf of $L_{q}^{\#}\in\F^{\#}$ is defined  as the set of points $x\in M$ that are the end point of a 
 piece-wise smooth horizontal geodesic starting at $q$. 

In \cite{Wilking-duality}  Wilking proved that  
Sharafutdinov retraction is smooth using the dual foliation of the metric retraction onto the soul. 
He also proved that if  the curvature is positive then
$M$ coincides with one single leaf of $\F^{\#}.$ In fact this result was proved in the more general context of singular
Riemannian foliations.

Dual foliations can also be seen from the point of view of   geometric  control theory. 
We can consider the set of smooth vector fields 
$\mathcal{C}=\{ \vec{f}_u \}_{u\in U}$ determined by the family
of  horizontal unit geodesic vector fields $\vec{f}_u$, i.e., those whose 
 integral curves are horizontal  unit speed  geodesic segments. With this approach, $L_{q}^{\#}$ is the attainable set $\mathcal{A}_{q}(\mathcal{C}).$
Since in Riemannian geometry this   controls system  is symmetric (if $\vec{f}_u\in\mathcal{C}$ then $-\vec{f}_u\in\mathcal{C}$),
 each orbit $\mathcal{O}(q)$ coincides with $\mathcal{A}_{q}(\mathcal{C})$, and in particular the leaf $L_{q}^{\#}$ is 
an immersed submanifold.

When we consider this discussion in the broader context of Finsler  geometry, we see a significant conceptual shift.
As illustrated in the simple Example \ref{example-submersion-non-compact},  $L_{q}^{\#}=\mathcal{A}_{q}(\mathcal{C})$ does not  need to be a submanifold, since 
it does not need to coincide with the orbit $\mathcal{O}(q)$. In fact the set of smooth vector fields 
$\mathcal{C}=\{ \vec{f}_u \}$ may no longer be symmetric.

In this paper, we investigate the  attainable  set $\mathcal{A}_{q}(\mathcal{C})$ of 
 the set of smooth vector fields $\mathcal{C}$ determined by the family
of  horizontal unit geodesic vector fields  
to the fibers $\F=\{\rho^{-1}(c)\}$ of a homogenous analytic  Finsler submersion $\rho: M\to B$.

\begin{theorem}
\label{theorem-dualizacao}
Let $\mu: G \times M \to M$ be an analytic  Finsler action on an analytic compact Finsler  manifold $M$. 
Assume that $G$ is compact and that the orbits are principal. 
Let $\rho: M \to M/G = B$ be the Finsler submersion describing the homogenous  foliation $\F=\{\rho^{-1}(c)\}$ and 
$\mathcal{C}=\{ \vec{f}_u \}_{u\in U}$ be the set of   horizontal unit geodesic vector fields associated to the submersion $\rho$. 
Then
\begin{enumerate}
\item[(a)] If the  orbit $\mathcal{O}(q)$ is embedded then it coincides with the attainable set $\mathcal{A}_{q}(\mathcal{C})$.
\item[(b)] If $(M,F)$ has  non negative flag curvature and  
the flag curvature at one point $q_0$ is  $K(q_0)>0$ then $\mathcal{A}_{q}(\mathcal{C})=\mathcal{O}(q)=M$ for each $q\in M$.
\end{enumerate}
\end{theorem}

With this nice  and simple result 
we hope to stress to an  
audience of mathematicians with  good knowledge on Riemannian geometry,  
how natural is the relation between geometric control theory
and Finsler geometry. In particular, we try to write this note  
in a self contained presentation.
Theorem \ref{theorem-dualizacao} also motivates natural questions that  could be explored
in future research. The first  one is 
how to generalize it, dropping for example the 
condition of homogeneity of the submersion (see Remark \ref{remark-preserving-volume-submanifold}) or even considering the  
most general situation, i.e, dual foliations 
of  singular Finsler foliation; see \cite{Alexandrino-Alves-Javaloyes-SFF}. 
A more involving question is whether  dual foliations could be used  to prove smoothness of class of 
Finsler submetries, as they were used in Wilking's work  in the Riemannian case; see \cite{Wilking-duality}.  

Item (b) of Theorem \ref{theorem-dualizacao}  follows direct from item (a) and from the Proposition \ref{proposition-Orbit-positive}
below, that assures  (as in the Riemannian case) that the orbits of
the set of    horizontal unit geodesic vector fields associated to a singular Finsler foliation
(e.g, partition of $M$ into orbits of a Finsler action) coincide  with $M$ when flag curvature is positive.  
 Note  nevertheless that Proposition \ref{proposition-Orbit-positive}  does not deal with attainable sets,  
that seems to us the main subject of study  in the Finsler case.

\begin{proposition}
\label{proposition-Orbit-positive}
Let $(M,F)$ be a complete Finsler manifold with non negative flag curvature and $\F=\{L\}$ a singular Finsler foliation.
Let $\mathcal{C}=\{ \vec{f}_u \}$ be the set of   horizontal unit geodesic vector fields of $\F$. 
Assume that  there exists a regular leaf $L_q$ so that 
each point of this leaf has positive flag curvature $K > 0$. 
Then the orbit $\mathcal{O}(q)$ of $\mathcal{C}$ coincides with $M$.

\end{proposition}

This paper is organized as follows:
In Section \ref{section-background} we review a few facts on geometric control theory and Finsler geometry that will be used in this paper. 
Item (a) of Theorem \ref{theorem-dualizacao} is proved in Section \ref{section-proof-Theorem-a}. 
Proposition \ref{proposition-Orbit-positive} (and hence item (b) of Theorem \ref{theorem-dualizacao})  
 is proved in Section \ref{section-prof-main-proposition}, 
accepting a few facts on Wilking's transversal Jacobi fields and the Jacobi triples in Finsler case, 
that are revised in  Section \ref{Wilking-transverse-Jacobi-fields}.
In particular, we hope that the review presented 
in  Section \ref{Wilking-transverse-Jacobi-fields} 
allow   Finslerian geometers
 to come into contact 
with a tool that has been  useful in the study of Riemannian submersions and (singular) Riemannian foliations.



\vspace{\baselineskip}

\section{Background}
\label{section-background}

\subsection{A few facts on geometric control theory}

Here we review a few results and definitions on geometric control theory 
extracted from the classical book of Agrachev and Sachkov \cite[Chapters 5, 8]{Agrachev-Sachkov}
and \cite{Stefan,Sussmann}.

Let $N$ be a manifold and  $\mathcal{C} = \{ \vec{f}_u \}$ be a set of smooth (analytic) vector fields \emph{everywhere defined,} i.e., the union
of the domains of elements of $\mathcal{C}$ is $N.$ This condition will be always used in this paper.

The \emph{attainable set} of the family $\mathcal{C}$ through $q$ is defined as:
$$\mathcal{A}_{q}(\mathcal{C})=\{ e^{t_k \vec{f}_k}\circ \cdots \circ e^{t_1 \vec{f}_1} (q),\,  t_{i} \geq 0, \,  k\in \mathbb{N},\,  \vec{f}_{i}\in\mathcal{C} \} $$
where $e^{t \vec{f}_i}$ is the flow of $\vec{f}_i\in \mathcal{C}$ in instant $t$. 
The \emph{orbit} of the family $\mathcal{C}$ through $q$ is 
$$\mathcal{O}(q)=\{ e^{t_k \vec{f}_k}\circ \cdots \circ e^{t_1 \vec{f}_1} (q),\,  t_{i}\in\mathbb{R},\,  k\in \mathbb{N},\,  \vec{f}_{i}\in\mathcal{C} \} $$

Orbits have  nice structures as we see in the next result.

\begin{theorem}[Nagano-Stefan-Sussmann]
For a given set of  vector fields $\mathcal{C}$ everywhere defined on a smooth manifold $N$, the 
partition $\{\mathcal{O}(q)\}_{q\in N}$ is a singular foliation, i.e, 
\begin{enumerate}
\item each orbit is an
immersed submanifold;
\item for each  $v_q\in T_{q}\mathcal{O}(q)$ there exists a vector field $\vec{v}$ on $N$ so that
$\vec{v}(q)=v_q$ and $\vec{v}(p)\in T_{p}\mathcal{O}(p), \forall p\in N.$
\end{enumerate}
\end{theorem}
Recall that when the leaves of a singular foliation have the same dimension, the singular foliation
is called \emph{regular foliation} or just \emph{foliation}.

Set $\mathrm{Lie}(\mathcal{C}) \, :=\mathrm{Span} \{ [\vec{f}_1,[\ldots, [\vec{f}_{k-1},\vec{f}_{k}]\ldots ] ], \vec{f}_{i}\in \mathcal{C}, k\in\mathbb{N} \}.$
With this concept we can establish conditions under which the orbit coincides with the manifold.

\begin{corollary}[Rashevsky-Chow]
\label{theorem5-9-agrachev}
Let $N$ be a connected manifold and $\mathcal{C}$ a set of vector fields. If $\mathrm{Lie}_{q}(\mathcal{C})=T_{q}N, \forall q\in N$,
then $N=\mathcal{O}(q), \forall q\in N$.
\end{corollary}

A submodule $\mathcal{V}$ (e.g., $\mathcal{V}=\mathrm{Lie}(\mathcal{C})$) 
is \emph{locally finitely generated over} $C^{\infty}(N)$, if for each point $q$, there exists a neigborhood $U$ of $q$ and 
vector fields $\vec{v}_{1}, \cdots, \vec{v}_{k}$ of $\mathcal{V}$ with domain containing $U$  so that  
$\mathcal{V}|_{U}=\{\sum_{i=1}^{k} a_{i} \vec{v}_{i} | a_{i}\in C^{\infty}(U)  \}.$ 

\begin{remark}
\label{remark-analicidade-localmente-finito}
\emph{if a module $\mathcal{V}$ is generated by analytic vector fields, it is locally finitely generated.}
This fact makes it possible to use relevant results on attainable  sets and it is  the main reason why we 
assume analicity in Theorem  \ref{theorem-dualizacao}.
\end{remark}

\begin{proposition} Let $N$ be a manifold and $q_0\in N$. 
If $\mathrm{Lie}(\mathcal{C})$ is locally finitely generated over $C^{\infty}(N)$, 
(in particular when $\mathcal{C}$ and $N$ are analytic) then $\mathrm{Lie}_{q}(\mathcal{C})=T_{q}\mathcal{O}(q_0)$ for 
$q\in\mathcal{O}(q_0)$ and for all orbits $\mathcal{O}(q_0).$
\end{proposition}

Different from orbits, the attainable sets do not need to be immersed submanifolds. But in the case where $\mathrm{Lie}(\mathcal{C})$ is 
locally finitely generated (e.g., $\mathcal{C}$ is analytic), they still have some interesting properties.

\begin{theorem}[Krener]
If $\mathrm{Lie}(\mathcal{C})$ is locally finitely generated, then $\mathrm{int}\big(\mathcal{A}_{q}(\mathcal{C})\big)$
is dense in $\mathcal{A}(q) \subset \mathcal{O}(q)$. Here the density is with respect to the topology of $\mathcal{O}(q)$. 
In particular $\mathrm{int} \big(\mathcal{A}_{q}(\mathcal{C})\big) \neq \emptyset$.
\end{theorem}

Let us now move towards results that will allow us to conclude that  $\mathcal{A}_{q_0}(\mathcal{C})=\mathcal{O}(q_0)$ 
under suitable hypotheses.

\begin{definition}
Given a complete vector field $\vec{\mathtt{g}}$ on $N$, 
a point $q\in N$ is called \emph{Poisson stable} for $\vec{\mathtt{g}}$
if for any $t_0 > 0$ and any neighborhood $W$ of $q$
there exists a point $x \in W$ and a time $t_1 > t_0$
such that $e^{t_{1} \vec{\mathtt{g}}}(x) \in W$.
The vector field $\vec{\mathtt{g}}$ is \emph{Poisson stable} 
if all points are Poisson stable.
\end{definition}

\begin{proposition}[Poincar\'{e}]
\label{proposition-prop-815-Poincare}
Assume that $N$ is compact and 
the flow $e^{t \vec{\mathtt{g}}}$ of a complete vector field $\vec{\mathtt{g}}$ preserves a volume of $N$. 
Then $\vec{\mathtt{g}}$ is Poisson stable.
\end{proposition}

\begin{definition}
A complete vector field $\vec{f}$ tangent to the orbits of $\mathcal{C}$ is called \emph{compatible with} $\mathcal{C}$ 
if $\mathcal{A}_{q}(\mathcal{C})$ is dense in $\mathcal{A}_{q}(\mathcal{C}\cup \vec{f})$, 
with respect to the topology of the orbits.
\end{definition}

\begin{proposition}
\label{proposition-prop-814-Agrachev}
Assume that $\mathrm{Lie}(\mathcal{C})$ is locally finitely generated.
If a complete vector field $\vec{\mathtt{g}}\in \mathcal{C}$ is Poisson stable in the orbit associated to $\mathcal{C}$, 
then $-\vec{\mathtt{g}}$ is compatible  with $\mathcal{C}$.
\end{proposition}

\begin{proposition}
\label{proposition-cor-83-Agrachev}
Assume that $\mathrm{Lie}(\mathcal{C})$ is locally finitely generated. 
If $\mathcal{A}_{q_0}(\mathcal{C})$ is dense in $\mathcal{O}(q_0)$ 
then $\mathcal{A}_{q_0}(\mathcal{C})=\mathcal{O}(q_0)$.
\end{proposition}

\vspace{0.5\baselineskip}

\subsection{A few facts on Finsler geometry}

In this section, we briefly review a few facts on Finsler geometry and Finsler submersions necessary to our paper, most of them 
 extracted from \cite{Alvarez-Duran}, \cite{Alexandrino-Alves-Javaloyes-SFF}, \cite{Alexandrino-Alves-Javaloyes-Equifocal} and
\cite{Foulon-Matveev}.  
A comprehensive introduction to this rich geometry can be found in \cite{book-Shen-lectures}.

\subsubsection{The metric structure}

\begin{definition}
Let $M$ be a manifold. A continuous function $F: TM \to [0,+\infty)$ is called \emph{Finsler metric} if
\begin{enumerate}
\item[(a)] $F$ is smooth on $TM \setminus \{0\}$,

\item[(b)] $F$ is positive homogeneous of degree $1$, that is, $F(\lambda v)=\lambda F(v)$ for every $v\in TM$ and $\lambda>0$,

\item[(c)] for every $p\in M$ and $v\in T_pM \setminus \{0\}$, the \emph{fundamental tensor} of $F$ defined as
$$g_{v}(u,w) = \frac{1}{2}\frac{\partial^{2}}{\partial t\partial  s} F^{2}(v+tu+sw)|_{t=s=0}$$
for any $u, w \in T_{p}M$ is a nondegenerate positive-definite bilinear symmetric form, i.e., an inner product.
\end{enumerate}
In particular, if $V$ is a vector space and $F: V \to \RR$ is a function smooth on $V \setminus \{0\}$ and satisfying the properties $(b)$ and $(c)$ above, then $(V, F)$ is called a \emph{Minkowski space}.
\end{definition}

The fundamental tensor satisfies a few relevant properties.

\begin{proposition}
For each $v \in TM \setminus \{0\}$ we have:
\begin{enumerate}
\item $g_{\lambda v}=g_{v}$ for all $\lambda>0$;

\item $g_{v}(v,v)=F^{2}(v)$;

\item $g_{v}(v,u) = \frac{1}{2}\frac{\partial}{\partial s} F^{2}(v+su)|_{s=0}=\mathcal{L}(v)u$ 
where $\mathcal{L}: TM \setminus \{0\} \to T^{*}M \setminus \{0\}$ 
is the Legendre tranformation;

\item $g_{v}(v,u) \leq F(v)F(u)$ for all $u \in T_{\pi(v)} M$.
\end{enumerate}
\end{proposition}

%

We can define the \emph{length} of a smooth piecewise curve $\gamma:[a,b]\to M$ as
$l_{F}(\gamma)=\int_{a}^{b} F(\gamma'(s)) ds$.
The \emph{distance} from $p$ to $q$  can be defined as
$d(p,q)=\inf_{\gamma\in \Omega_{p,q}} l_{F}(\gamma)$,
where $\Omega_{p,q}$ is the set of curves 
$\gamma:[0,1]\to M$ joining $p=\gamma(0)$ to 
$q=\gamma(1)$. Unlike Riemannian geometry,  the distance $d(p,q)$ does not need to
be equal to the distance $d(q,p)$. 
But we can  still  have  several important metric geometric concepts from Riemannian geometry, 
as long as we take into account    the orientations of the curves involved in the definition of the distances.
For example, instead of talking about a metric ball, now we have to 
talk about \emph{future balls}, i.e. $B_{r}^{+}(p)=\{x\in M \,|\, d(p,x)<r \}$, and \emph{past balls}, i.e. $B_{r}^{-}(p)= \{x\in M \,|\, d(x,p)< r \}$, 
 
Since we have a length  functional  on the space of smooth piecewise  oriented curves, we can define a geodesic as an 
oriented curve that locally minimizes the distance. More precisely a curve $\gamma:[a,b]\to M$ is called \emph{geodesic} if
for each $s_0\in [a,b]$ there exists $\epsilon>0$ so that $d(\gamma(s_0),\gamma(s)) = \int_{s_0}^{s}F(\gamma'(t)) dt$ 
where $s\in[s_0, s_0+\epsilon].$ Just like in Riemannian geometry, geodesics can also be seen as critical points of energy functional 
$\gamma\to \int_{a}^{b} F^{2}(\gamma'(s))ds$ or as curves with zero accelerations  with respect to the (Chern) covariant derivative. But before we start to 
review the concept of Chern connection, let us end this subsection with a concrete important example.

\begin{example}[Randers metric]
Let $h$ be a Riemannian metric and $\vec{w}$ be a smooth vector field with $\|\vec{w}\|<1$, where $\|\vec{w} \|= h(\vec{w},\vec{w})^{1/2}$. 
We define the Randers metric $F$ with Zermelo data $(h,\vec{w})$ by the intrinsic equation:
$$ \|v-F(v)\vec{w} \| = F(v).$$
In other words, $\mathcal{I}_{p}^{F}(\epsilon)=\mathcal{I}_{p}^{h}(\epsilon)+ \epsilon \vec{w}(p)$ where the \emph{indicatrix} $\mathcal{I}_{p}^{F}(\epsilon)$
is defined as $\{v\in T_{p}M \,|\, F(v)=\epsilon \}.$ 
The Randers metric $F$ can also be defined as $F=\alpha+\beta$ where $\alpha$ is a Riemannian norm and $\beta$ is a 1-form so that 
$\|\beta\|_{\alpha}<1$. There is a bijection between $(h, \vec{w})$ and $(\alpha, \beta)$, but we will not need it in this paper.
\end{example}

We are interested in two properties of geodesics in Randers manifolds that we formulated as follows:


\begin{proposition}
Let $F$ be a Randers metric with Zermelo data $(h,\vec{w}),$ where $\vec{w}$ is a Killing vector field on $M$ with respect to $h$.
Let $\gamma$ be a unit speed geodesic with respect to $h$.
\begin{enumerate}
\item Then $t\to \beta(t)= e^{t\vec{w}}\circ \gamma(t)$ is a unit speed geodesic with respect to the Randers metric $F$.

\item If $\gamma$ is a unit speed geodesic starting orthogonal to a submanifold $L$ with respect to the Riemannian metric $h$, 
i.e $h(\gamma'(0),v)=0$ for all $v\in T_{\gamma(0)}L$, 
then the unit geodesic $\beta$ (with respect to $F$) is orthogonal to $L$ with respect to $g_{\beta'}$, i.e., $g_{\beta'(0)}(\beta'(0),v)=0$
for all $v\in T_{\beta(0)}L$.
\end{enumerate}
\end{proposition}
\vspace{0.25\baselineskip}

\begin{remark}
There is also an easy way to produce 
Finsler actions on Randers spaces. 
An action $\mu:G\times M\to M$  is a Finsler action
(i.,e $F(d\mu)=F$) on Randers spaces
with Zermelo data $(h,\vec{w})$ if and only if 
the action is isometric (with respect to $h$)
and $\vec{w}$ is $G$ invariant, i.e, 
$ \vec{w}\circ\mu_{g} = d \mu_{g}\vec{w}$.
\end{remark}

\subsubsection{Chern connection and Jacobi fields}

Let us now  review the concept of  Chern connection associated with a Finsler metric $F$ 
as a family of affine connections.


\begin{proposition}
\label{proposition-chern-connection}
Given a vector field $\vec{v}$ without singularities on
an open set $U\subset M$, there exists a unique 
affine connection $\nabla^{v}$ on $U$ (the so called Chern connection) that satisfies
the following properties:
\begin{enumerate}
\item[(a)] $\nabla^{v}$ is \emph{torsion-free}, namely,
$$\nabla^{v}_{\vec{f}}\,\vec{g}-\nabla^{v}_{\vec{g}}\, \vec{f}=[\vec{f},\vec{g}]$$
for every vector fields $\vec{f}$ and $\vec{g}$ on $U$,

\item[(b)] $\nabla^{v}$ is \emph{almost g-compatible}, namely,
$$\vec{f}\cdot g_{v}(\vec{g},\vec{w})=g_{v}(\nabla^{v}_{\vec{f}}\, \vec{g},\vec{w})+g_{v}(\vec{g},\nabla^{v}_{\vec{f}}\, \vec{w})
+2 \, C_{v}(\nabla^{v}_{\vec{f}} \, \vec{v},\vec{g},\vec{w}).$$
\end{enumerate}
Here $C_v$ is the \emph{Cartan tensor} associated with the Finsler metric defined as:
\begin{eqnarray*}
C_{v}(w_{1},w_{2},w_{3}) & := & \frac{1}{2}\frac{\partial}{\partial s} g_{v+sw_{1}}(w_2,w_3)|_{s=0}\\
& = & \frac{1}{4} \frac{\partial^{3}}{\partial s_{3}\partial s_{2}\partial s_{1}} F^{2}(v+ \sum_{i=1}^{3} s_{i}w_{i})|_{s_{1}=s_{2}=s_{3}=0}
\end{eqnarray*}
for every $p \in M$, $v \in T_p M \setminus \{0\}$, and $w_{1}, w_{2}, w_{3} \in T_p M$.
\end{proposition}

It can be checked that the Christoffel symbols of $\nabla^v$ only depend on $v=\vec{v}(p)$ at every $p\in M$, and not on the particular extension. Therefore, the Chern connection is an anisotropic  connection. 
Moreover, it   is positively homogeneous of degree zero, namely, $\nabla^{\lambda v}=\nabla^v$ for all $v\in  TM \setminus \{0\}$ and $\lambda>0$. 
One can also prove the following property of Cartan tensor:
\begin{equation}
\label{eq-propriedade-tensor-Cartan}
C_{v}(v,w_1,w_2)=C_{v}(w_1,v,w_2)=C_{v}(w_1,w_2,v)=0.
\end{equation}

Let $\gamma: I\subset \mathbb{R}\to M$  be a piecewise smooth curve 
and $t\to \vec{w}(t)$  a vector field without singularities along $\gamma$, i.e, $\vec{w}$ is a section of the pullback 
fiber bundle $\gamma^{*}(TM)$ over $I$. By considering the pullback of the Chern connection $\nabla^{w}$ we induce
the covariant derivative $\frac{\nabla^{w}}{d t}$ along $\gamma$. In particular we have that $\frac{\nabla^{w}}{d t} \, \vec{f}(t)=\nabla_{\gamma'}^{w}\, \vec{f}$
when $\vec{f}\in \mathfrak{X}(M)$.

Now we can given an equivalent definition of geodesic. A smooth curve $\gamma: I\subset \mathbb{R}\to M$ is a geodesic if and only if
$\frac{\nabla^{\gamma'}}{dt}\gamma'(t)=0$.


A  geodesic   can also been seen as the projection of an integral curve 
 of the (Finsler) geodesic spray. In other words, we have a  vector field $\vec{\mathtt{g}}$ (\emph{the Finsler geodesic spray})
on $TM-\{0\}$ so that the geodesic $\gamma$ with initial condition $\gamma'(0)=v_p\in T_{p}M$ is 
$\gamma(t)=\pi\big(e^{t \vec{\mathtt{g}} }v_p \big),$ where $\pi:TM\to M$ is the canonical projection.
We say that $(M,F)$ is a \emph{complete Finsler
manifold}, if the Finsler geodesic spray $\vec{\mathtt{g}}$ is a complete
vector field, i.e., its integral curves are defined for all  $ t\in \mathbb{R}$. 
The Finsler geodesic spray $\vec{\mathtt{g}}$ has also the interesting property that \emph{its flow preserves a volume form} 
$\omega$ (the so called volume of the \emph{Sasaki metric}) on $TM-\{0\}$, see \cite[Propositions 5.4.2, 5.4.3]{book-Shen-lectures}.


When we consider a geodesic variation $t\to\gamma_s(t)=\gamma(s,t)$  of a geodesic $\gamma$, 
then the variational vector field $J(t)=\frac{\partial}{\partial s}\gamma(0,t)$, the so called \emph{Jacobi vector field along $\gamma$}, 
is characterized by solving the differential equation
\begin{equation}\label{jacobi}
J'' (t)+ R_{\dot{\gamma}(t)}(J(t))=0.
\end{equation}
Here $J'(t)=\frac{\nabla^{\gamma'}}{d t} J$ and $R_v:T_pM\rightarrow T_pM$ is an operator well defined for each $p\in M$ and $v\in T_pM \setminus \{0\}$ called
\emph{Jacobi operator}. It can be well defined by properties of isotropic connections; see \cite[Section 5]{Alexandrino-Alves-Javaloyes-Equifocal}.

The \emph{flag curvature} for $v \in TM \setminus \{0\}$ and
$w\in T_{\pi(v)}M$ is defined in analogous way to the  sectional curvature in Riemannian case.
$$K(v,w)=\frac{g_{v}(R_{v}w,v)}{g_{v}(v,v)g_{v}(w,w)-g_{v}(v,w)^{2}}.$$

\begin{remark}
There is a natural way to produce Finsler spaces with non negative or positive flag curvature. Given a Riemannian
manifold $(M,h)$ with nonnegative or positive  sectional curvature, the Randers space $(M,F)$ with Zermelo 
data $(h,\vec{w})$ (where $\vec{w}$ is a Killing field of $(M,h)$) has nonnegative or positive flag curvature. 
\end{remark}

The next proposition provides us with a natural relationship between these well-known concepts in Riemannian geometry and their analogues in Finsler geometry.

\begin{proposition}
\label{proposition-covariantderivative-finsler}
Let $(M,F)$ be a complete  Finsler manifold and 
$\vec{v}$ be a geodesic vector field on an open subset $U \subset M$, 
let $\hat{g}:= g_{\vec{v}}$ denote the Riemannian metric on $U$ induced by the fundamental tensor $g$ 
and let $\widehat{\nabla}$ and $\widehat{R}$ be the Levi-Civita connection and the Jacobi operator of $\hat{g}$, respectively. 
Then, for any $\vec{f} \in {\mathfrak X}(U)$,
\begin{enumerate}
\item $\widehat{\nabla}_{\vec{f}}\, \vec{v}=\nabla^{v}_{\vec{f}}\, \vec{v}$ and $\widehat{\nabla}_{\vec{v}}\, \vec{f}=\nabla^{v}_{\vec{v}}\, \vec{f}$,

\item $\widehat{R}_{\vec{v}}\, \vec{f}=R_{\vec{v}}\, \vec{f}$.
\end{enumerate}
As a consequence, the integral curves of $\vec{v}$ are also geodesics of $\widehat{g}$, 
and the Finslerian Jacobi operator and Jacobi fields along the integral curves of $\vec{v}$ coincide with those of $\widehat{g}$.
\end{proposition}

We finish this subsection by recalling the concept of $L$-Jacobi fields. 

\begin{definition}
\label{definition-jacobi-field}
Let $L$ be a submanifold of a complete Finsler manifold $(M,F)$ and $\gamma:[a,b]\to M$ a unit speed geodesic orthogonal
to $L$ at $p=\gamma(a)$. We say that a Jacobi field $J$ is a $L$-Jacobi field if
\begin{itemize}
\item $J(a)$ is tangent to $L$;

\item $\mathcal{S}_{\gamma'(a)}J(a)= \mathrm{tan}_{\gamma'(a)}J'(a)$ where
$\mathcal{S}_{\gamma'}:T_{p}L\to T_{p}L$
is the \emph{shape operator} defined as
$\mathcal{S}_{\gamma'}(u) = \mathrm{tan}_{\gamma'(a)}\nabla_{u}^{\gamma'(a)}\xi$
with $\xi$ an orthogonal vector field along $L$ such that
$\xi_{p}=\gamma'(a)$ and $\mathrm{tan}_{\gamma'(a)}$
is the $g_{\gamma'(a)}$-orthogonal projection onto $T_{p}L$.
\end{itemize}
\end{definition}

\begin{remark}
\label{remark-L-jacobi}
As proved in   \cite[Proposition 3.5]{Alexandrino-Alves-Javaloyes-Equifocal}, 
a Jacobi field $J$ along $\gamma$ is a $L$-Jacobi field
if and only if it is variation vector field of
a variation of $L$ orthogonal geodesics.
\end{remark}

\subsubsection{Finsler submersion}

\begin{definition}
A submersion $\rho:(M,F)\rightarrow (B,F^{\star})$  between Finsler manifolds is a \emph{Finsler submersion} if 
$d\rho_p(B^F_p(0,1))=B^{F^{\star}}_{\rho(p)}(0,1),$ for every $p\in M$, where $B^F_p(0,1)$ and $B^{ F^{\star}}_{\rho(p)}(0,1)$ 
are the unit balls of the Minkowski spaces $(T_pM,F_p)$ and $(T_{\rho(p)}B, F^{\star}_{\rho (p)})$ centered at $0$, respectively.
\end{definition}

The first natural example is to consider a Finsler action $\mu:G\times M\to M$ (i.e, $F (d \mu_g)= F$) 
where all orbits have the same isotropy type, i.e, the isotropy groups $G_p=\{g \in G \,|\, \mu(g,p)=p \}$ are conjugate to each other. 
Then the projection $\rho:(M,F)\to (M/G,F^{\star})$ is a Finsler submersion where $F^{\star}$ is the induced Finsler norm on $B=M/G$.

It is also useful to construct Finsler submersions in Randers spaces starting with a Riemannian submersion.

\begin{lemma}
\label{lemma-ex-finsler-submersion}
Let $\rho:(M,h)\to (B,h^{\star})$ be a Riemannian
submersion, $\vec{w}^{\star}$ a vector field on $B$ 
and $\vec{w}$ a vector field in $M$ that is $\rho$-related to $\vec{w}^{\star}$, i.e., $d \rho \circ \vec{w}= \vec{w}^{\star}\circ \rho$.
Then $\rho:(M, R)\to (B, R^{\star})$ is
a Finsler submersion, where  $R$ is the Randers metric with Zermelo data $(h,\vec{w})$ 
and $R^{\star}$ is the Randers metric with Zermelo data $(h^{\star}, \vec{w}^{*})$.
\end{lemma}

Given the Finsler foliation $\F=\{L \}$ with leaves $L=\rho^{-1}(c)$, 
we say that a geodesic $\gamma:I\subset \RR\rightarrow M$ is \emph{horizontal} if  for each $t\in I$ the vector 
$\gamma'(t)$  is an orthogonal to  the leaves $L\in\mathcal{F}$, i.e., 
$g_{\gamma'(t)}(\gamma'(t),w)=0$ for all $w\in  T_{\gamma(t)}L$.


\begin{center}
\begin{figure}[tbp]
	\includegraphics[scale=0.4]{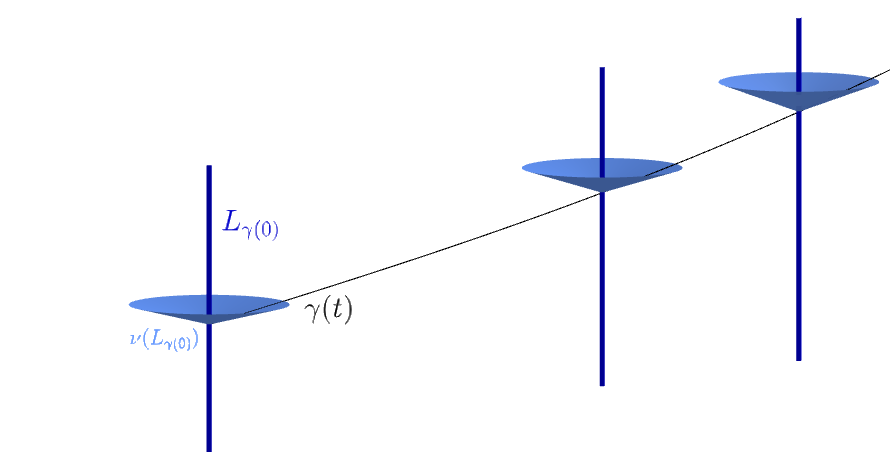}
	\caption{Figure generated by the 
software geogebra.org  illustrating  Lemma \ref{lemma-ex-finsler-submersion}, i.e., a Randers submersion that was produced
	starting with the trivial Riemannian submersion $\rho:\mathbb{R}^{3}\to\mathbb{R}^{2}$ defined as 
	$\rho(x)=(x_1,x_2)$ and wind $W=(0,0,\frac{1}{4}\sin^{2}(x_1)+\frac{1}{4})$. The horizontal unit speed 
		geodesic is $\gamma(t)=(t,0,\frac{3t}{8}-\frac{\sin(2t)}{16}).$ Note that the union of normal vectors to a  tangent space of the fibers 
	is a (normal) cone, and no longer a normal subspaces as it was in the Riemannian case. As remarked the 
geodesics of a Finsler submersion are orthogonal to the leaves and hence tangent to the \emph{normal cones} \textcolor{Blue}{$\nu(L)$}. }
	\label{figura-submersaoRanders}
\end{figure}
\end{center}

In the same way as in Riemannian geometry, in Finsler geometry we have the lift property of  geodesics.

\begin{proposition}
\label{proposition-duran-geodesics-submersion}
Let $\pi:(M,F)\to (B, F^{\star})$ be a Finsler submersion. 
Then an immersed curve on $B$ is a geodesic if and only if
its horizontal lifts are geodesics on $M$. In particular, the geodesics of $(B, F^{\star})$ are precisely the projections of horizontal
geodesics of $(M,F)$.
\end{proposition}

Once we fix a geodesic vector field, we can reduce the study of Finsler submersions to Riemannian submersions.

\begin{proposition}
Let $\pi:(M,F)\to (B, F^{\star})$ be a Finsler submersion. 
Let $v^\star$ be a geodesic vector field in some open subset $U^{\star}$ of $B$. 
Then the horizontal lift $\vec{v}$ of $v^*$ is a geodesic vector field on $U=\rho^{-1}(U^{\star})$ 
and the restriction $\rho|_U:(U,g^F_{\vec{v}})\rightarrow ( U^{\star},g^{F^\star}_{v^*})$ is a Riemannian submersion, 
where $g^F$ and $g^{F^\star}$ are the fundamental tensors of $F$ and $ F^\star$, respectively.
\end{proposition}

We finish this subsection by presenting two simple examples illustrating how compactness hypothesis of Theorem \ref{theorem-dualizacao} is important.

\begin{example}
\label{example-submersion-non-compact}
In this example we present an attainable set and orbit of the set 
of  horizontal unit geodesic vector fields of a Finsler homogenous analytical submersion on a non compact space, see Figure \ref{figura-Alcancavel-nao-compacto}.
Consider the Riemannian submersion $\rho: (\mathbb{R}^{2},h_{2}) \to (\mathbb{R},h_{1})$ where 
$\rho(x_1,x_2)=x_1$ and $h_{n}$ is the Euclidean metric on $\mathbb{R}^{n}$.
By Lemma \ref{lemma-ex-finsler-submersion}, taking $\vec{w}=(0,\frac{1}{2})$,  we have that 
$\rho: (\mathbb{R}^{2},R) \to (\mathbb{R}, h_1)$ is a Finsler submersion 
where $R$ is the Randers metric with respect to the Zermelo data $(h_2,\vec{w})$. 
Let $\mathcal{C} = \{\vec{f}_{1}, \vec{f}_{2} \}$ be the set of vector fields where $\vec{f}_{1} = (1,\frac{1}{2})$ and $\vec{f}_{2}=(-1,\frac{1}{2})$. 
The integral curves of these vector fields are horizontal geodesics of the Finsler submersion 
$\rho: (\mathbb{R}^{2}, R) \to (\mathbb{R}, h_1)$.
It is easy to see that $\mathcal{A}_{(0,0)}(\mathcal{C})$ is a cone with its interior. 
More precisely $\mathcal{A}_{(0,0)}(\mathcal{C}) = \{x \in \mathbb{R}^{2} \,|\, \frac{1}{2} x_{1} \leq x_{2}, 0 \leq x_1 \} \cup \{x \in \mathbb{R}^{2} \,|\, -\frac{1}{2} x_{1} \leq x_{2}, x_{1} \leq 0 \}$.
Also it is clear that $\mathcal{O}\big((0, 0)\big) = \mathbb{R}^{2}$. In particular $\mathcal{O}\big((0,0)\big) \neq \mathcal{A}_{(0,0)}(\mathcal{C})$.
\end{example}


\begin{center}
\begin{figure}[tbp]
	\includegraphics[scale=0.4]{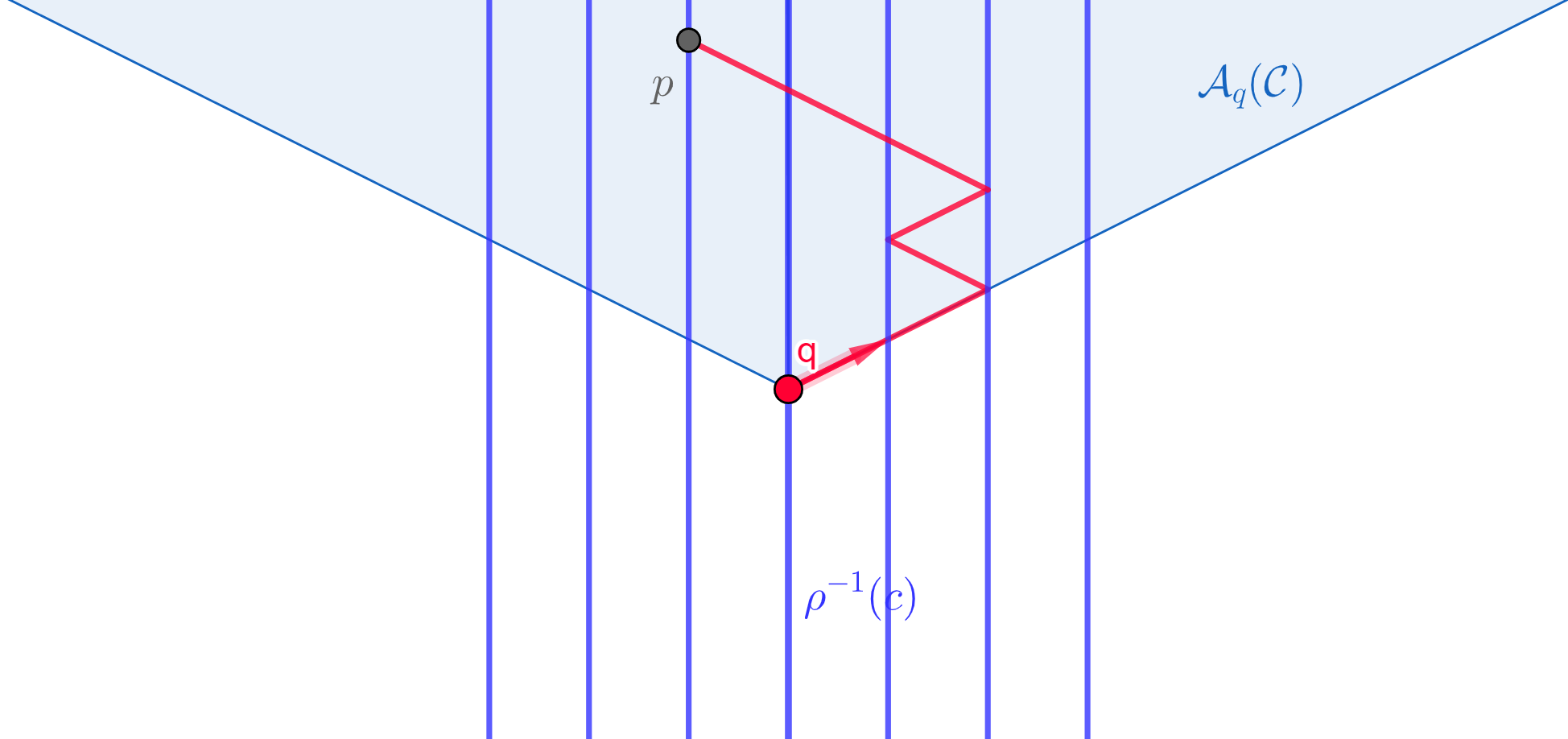}
	\caption{Figure generated by the 
software geogebra.org  illustrating Example \ref{example-submersion-non-compact}}
	\label{figura-Alcancavel-nao-compacto}
\end{figure}
\end{center}


\begin{example}
\label{example-submersion-compact}
Let $\pi: \mathbb{R}^{2} \to \mathbb{T}^{2}=\mathbb{R}^{2}/(\mathbb{Z}\times\mathbb{Z})$ be the canonical projection of the Euclidean plane onto the torus 
and the Finsler submersion $\rho: (\mathbb{R}^{2},R) \to (\mathbb{R}, h_1)$ defined in Example \ref{example-submersion-non-compact}. 
Since the fibers of the submersion $\rho$ and the vector field $\vec{w}$ are invariant by the action of $\mathbb{Z}\times\mathbb{Z}$, 
the Finsler submersion $\rho$ projects to a Finsler submersion $\rho^{\star}:(\mathbb{T}^{2}, R^{\star})\to (\mathbb{S}^{1},h_1)$. 
Here  $R^{\star}$ is the Randers metric with Zermelo data $(h_2,w^{\star})$ where the $h_2$ is flat metric and the wind $w^{\star}$ is $\rho$-related with $\vec{w}$, 
i.e., $w^{\star}\circ\pi=d \pi \circ \vec{w}$. Define $f_{i}^{\star}$ to be $\pi$-related to $\vec{f}_{i}$. 
The integral curves of $f_{i}^{\star}$ are horizontal geodesics of $\rho^{\star}$. 
Set $\mathcal{C}^{\star}=\{f_{1}^{\star}, f_{2}^{\star}  \}$ and $p^{\star}=\pi( (0,0))$. 
Then it is not difficult to check that $\mathcal{A}_{p^{\star}}(\mathcal{C}^{\star})=\mathcal{O}(p^{\star})$.

\end{example}

\vspace{\baselineskip}

\section{Proof of item (a) of Theorem \ref{theorem-dualizacao}}
\label{section-proof-Theorem-a}

Let $N \subset T^{1}M$ be the union of unit cone bundle of the fibers of $\rho$, i.e., $N := \cup_{x\in M}\nu_{x}^{1}(L_x)$ for $L_x = \rho^{-1}(\rho(x)).$ It follows from Alvarez Paiva and Duran \cite{Alvarez-Duran} that $N$ is a compact embedded submanifold of the unit bundle $T^{1}M$ and that the diagram below comutes

\vspace{0.5\baselineskip}
\begin{center}
\begin{tikzcd}[column sep={2.5cm}, row sep={1.5cm}]
N \arrow[swap]{d}{\pi_M} \arrow{r}{\rho_N} & T^1 B \arrow{d}{\pi_B} \\
M \arrow[swap]{r}{\rho} & B
\end{tikzcd}
\end{center}
\vspace{0.5\baselineskip}
where $\rho_{N} = d\rho |_N$ and $\pi_M$ and $\pi_B$ are the canonical projections.
Also note that $N$ is invariant by the geodesic flow $e^{t\vec{\mathtt{g}}}$ and 
\begin{equation}
\label{eq-comutation-flowgeodesic}
\rho_{N}\circ e^{t\vec{\mathtt{g}}} =e^{t\vec{\mathtt{b}}}\circ\rho_{N}
\end{equation}
where $e^{t\vec{\mathtt{b}}}$ is the geodesic flow in $T^{1}B$.

\begin{remark}
Note that the action
$\mu:G\times M\to M$ induces
an action $\tilde{\mu}:G\times N\to N$
as $\tilde{\mu}_{g}=(\mu_{g})_{*}$
and the orbits of the action induces
the leaves of the foliation 
$\widetilde{\F}:=\{\widetilde{L} \}$
where $\widetilde{L}=\rho_{N}^{-1}(c).$
\end{remark}

\begin{lemma}
\label{lemma-poisson-stable}
The Finsler  geodesic spray $\vec{\mathtt{g}}$ restricted to $N$ is Poisson stable.
\end{lemma}
\begin{proof}

In order to prove that the flow is Poisson stable, it suffices to check that
\begin{equation}
\label{Eq-1-lemma-poisson-stable}
|e^{t\vec{\mathtt{g}}}(W)|= |W|, \quad \forall t>0,
\end{equation}
where $W$ is any given proper relative compact neighborhood and $| \cdot |$ is a fixed volume on $N$ that will be constructed below,  recall Proposition \ref{proposition-prop-815-Poincare}.

The first step in our construction is to define a metric on the fibers of $\rho:M\to B=M/G$ so 
 that the end point map $\eta_{\xi}:G(x) \to G(y)$, defined as $\eta_{\xi}(x)=\exp_{x}(\xi)$, turns to be an isometry. 

Since all orbits are principal, 
the slice theorem implies that 
the map $ \Psi: x\to \mathfrak{g}_x\subset Gk(\mathfrak{g}) $ is smooth (where $Gk(\mathfrak{g})$ is the Grassmannian of $\mathfrak{g}$). 
For a given metric $\langle \cdot, \cdot \rangle $ on $\mathfrak{g},$ we can find a  subspace $V_x$
orthogonal to $\mathfrak{g}_x$, i.e., $\mathfrak{g}=V_x\oplus \mathfrak{g}_x$. 
Now we define the metric $\tilde{g}$ along the orbits  that
transform the isomorphism  $d\mu_{x}:(V_x, \langle \cdot, \cdot \rangle )  \to 
(T_{x}G(x),\tilde{g}_{x})$ into an isometry.

Since the isotropy groups along (minimal) horizontal geodesics are the same, 
we have that $V_x =V_y$, where $y=\eta_{\xi}(x)$. 
Note that $d \eta_{\xi}\vec{v}(x)=\vec{v}(y)$ 
where $\vec{v}$ is the vector field along the orbits defined as 
$\vec{v}(x)=d\mu_x v$ for $v\in V=V_x=V_y$. 
These facts allow us to conclude that the map $\eta_{\xi}:G(x) \to G(y)$ is an isometry.

Now we can define a volume form $\omega_G$ (with respect to the metric $\tilde{g}$) along the fibers of $\rho:M\to B=M/G$. 
Note that this form is invariant by the end point map, i.e., $\eta_{\xi}^{*}\omega_{G}=\omega_G$. 
The metric $\tilde{g}$ as well the volume form $\omega_G$ 
can also be defined on the fibers $\{\widetilde{L}\}$ of $\rho_{N}:N\to T^{1}B$ and we will use the same notation. 
Note that if $e^{t\vec{\mathtt{g}}}:\widetilde{L}_x\to \widetilde{L}_y$, 
then $(e^{t\vec{\mathtt{g}}})^{*}\omega_{G}=\omega_{G}$.
 By another abuse of notation, consider $\omega_{G}$ an extension of the previous $\omega_{G}$ to a $k$-form in $N$.

We can define the volume form as $\omega=\omega_{G}\wedge \rho_{N}^{*}\omega_B$ 
where $\omega_B$ is the volume form with respect to the Sasaki metric of $T^{1}B$. 
Recall that $e^{t \vec{\mathtt{b}}}$ preserves the volume form for the Sasaki metric of $T^{1}B$. 
Therefore we conclude that
\begin{equation}
\label{Eq-2-A-lemma-poisson-stable}
(e^{t\vec{\mathtt{g}}})^{*}\omega=\omega.
\end{equation}

Defining $|W|=\int_{W} \omega$, we conclude that Eq.\eqref{Eq-2-A-lemma-poisson-stable} implies Eq. \eqref{Eq-1-lemma-poisson-stable} 
and this conclude the proof of the lemma.
\end{proof}

\begin{remark}
\label{remark-preserving-volume-submanifold}
Note that the fact that  Finsler geodesic spray  preserves volume  on $TM$  does not directly imply that
its restriction to a submanifold of $TM$ preserves volume of this submanifold. This is one
of the reasons why we are assuming that the submersion $\rho:M\to B$ is homogenous. 
\end{remark}

Now we are going to define a set of vector fields $\widetilde{\mathcal{C}}$ on $N$
that will be relates to our original set of vector fields $\mathcal{C}$ on $M$ as follows:
\begin{enumerate}
\item[(i)] $\pi_{M}\big(\widetilde{\mathcal{A}}_{\tilde{q}}(\widetilde{\mathcal{C}})\big)= \mathcal{A}_{q}(\mathcal{C})$, 
\item[(ii)] $\pi_{M}\big(\widetilde{\mathcal{O}}(\tilde{q})\big)= \mathcal{O}(q)$, 
\end{enumerate}
where $\pi_{M}(\tilde{q})=q$. Therefore, once we have proved that 
$\widetilde{\mathcal{A}}_{\tilde{q}}(\mathcal{C})= \widetilde{\mathcal{O}}(\tilde{q})$ (see Eq. \eqref{eq-fim-prova-itema} below)
we will be able to conclude that $\mathcal{A}_{q}(\mathcal{C})= \mathcal{O}(q)$ and hence to finish the proof 
of item (a) of  Theorem \ref{theorem-dualizacao}.

Consider a set of  vector fields $\mathcal{C}_1=\{\vec{f}_u\}$ with the following properties: 
 $\vec{f}_u $ are tangent to the fibers of $\pi_{M}:N\to M$,  $\mathcal{C}_1$ is symmetric
(i.e. if $\vec{f}_u\in\mathcal{C}_1$ then $-\vec{f}_{u}\in\mathcal{C}_1$) 
and $\vec{f}_u $  are $\mu_{g}$-related, i.e., $\vec{f}_{u}\circ\tilde{\mu}_{g}= d\tilde{\mu}_{g} \vec{f}_{u}$. 
Now we complete $\mathcal{C}_1$ with the  geodesic spray (restricted to $N$),
i.e., $\widetilde{\mathcal{C}}=\{\vec{\mathtt{g}} \}\cup \mathcal{C}_1$. Note that the projection  of 
 the integral curves of $\vec{f}_{0}=\vec{\mathtt{g}}$ 
corresponds to the horizontal geodesics and the projection of  the integral curves of $\vec{f}_u$ ($u\neq 0$) measure how it breaks, and in particular  
$\pi_{M}(\gamma_{2}*\delta_{u}*\gamma_{1})$ is a broken horizontal unit speed  geodesic, where 
$\gamma_{i}$ is a integral line of $\vec{\mathtt{g}}$ and $\delta_{u}$ is an integral line of $\vec{f}_{u}.$
The attainable set  and the orbit of the family
$\widetilde{\mathcal{C}}$
through $\tilde{q}_0$ are denoted by 
$\widetilde{\mathcal{O}}(\tilde{q}_{0})$ and
$\widetilde{\mathcal{A}}_{\tilde{q}_{0}}(\widetilde{\mathcal{C}}).$
Using this one can check properties (i) and (ii) stated above.

\begin{lemma}
\label{lemma-poisson-stable-na-orbita}
Consider $\widetilde{\mathcal{O}}(\tilde{q})$ the orbit of the family $\widetilde{\mathcal{C}}$ through $\tilde{q}$ defined above.
Then
\begin{enumerate}
\item[(a)] Each orbit $\widetilde{\mathcal{O}}(\tilde{q})$ meet all the fibers of $\rho_{N}$.
\item[(b)] The leaves of the foliation of orbits $\{\widetilde{\mathcal{O}}(\tilde{q}) \}$ has the same dimension (i.e, $\{\widetilde{\mathcal{O}}(\tilde{q})\}$
is a regular foliation). 
\item[(c)] If the orbits $\mathcal{O}(q) $ in $M$ are embedded then the orbits $ \widetilde{\mathcal{O}}(\tilde{q}) $ in $N$ are embedded as well.
\item[(d)] If the orbits $\widetilde{\mathcal{O}}(\tilde{q}) $ in $N$ are embedded, then $\vec{\mathtt{g}}$ restricts to each orbit is Poisson stable.
\end{enumerate}
\end{lemma}

\begin{proof} \quad

\emph{Item (a)} Fix a point $\tilde{q}_{0}\in N$ and consider  $\tilde{q}_{1}\in N.$ 
Let $\gamma_{B}:[0,r]\to B$ be a  piecewise broken unit geodesic 
 so that $\gamma_{B}'(0) = \rho_{N}(\tilde{q}_{0})$ and $\gamma_{B}'(r) = \rho_{N}(\tilde{q}_{1}).$
Let  $\gamma_N:[0,r]\to N$ be the lift of $\gamma_{B}$ with $\gamma_{N}'(0) = \widetilde{q}_{0}.$
Set $\gamma_{N}'(r) = \tilde{q}_{2}.$ Note that $\rho_{N} (\tilde{q}_{2})= \rho_{N} (\tilde{q}_{1})$ and hence that
$\widetilde{q}_{2}\in \widetilde{L}_{\tilde{q}_{1}}$. 
Since $\gamma_N \subset \widetilde{\mathcal{O}}(\tilde{q}_0)$, 
we have concluded that 
$\tilde{q}_2 \in \widetilde{\mathcal{O}}(\tilde{q}_0) \cap \widetilde{L}_{\tilde{q}_1}$ 
and this finishes the proof of item (a).

\vspace{0.25\baselineskip}

\emph{Item (b)} Recall that the action $\mu:G\times M\to M$ induces an action $\tilde{\mu}: G\times TM \to TM$
as $\tilde{\mu}_{g}=(\mu_{g})_{*}$ and the orbits of the action $\tilde{\mu}$ induces the leaves of $\widetilde{\mathcal{F}}.$ 
For a fix  leaf $\widetilde{L}_{\tilde{q}}$, note that 
$\widetilde{\mu}_{g} (\widetilde{\mathcal{O}}(\tilde{q})) = \widetilde{\mathcal{O}}(\widetilde{\mu}_{g} (\tilde{q})).$
Since $\widetilde{\mu}_{g}$ is a diffeomorphism, the orbits  that meet $\widetilde{L}_{\tilde{q}}$ 
 have the same dimension. On the other hand,
it follows from item (a) that all orbits meet  $\widetilde{\mathcal{F}}$. Therefore all orbits have the same dimension.

\vspace{0.25\baselineskip}



 \emph{ Item (c)} 
Since $\pi_M$ is a submersion (consequently transverse to $\mathcal{O}(p)$) it suffices to prove that 
$\pi_{M}^{-1}(\mathcal{O}(p)) = \widetilde{\mathcal{O}}(u_p)$, for all $u_p \in N$. 
The inclusion $\pi_{M}^{-1}(\mathcal{O}(p)) \supset \widetilde{\mathcal{O}}(u_p)$ follows immediately from the fact that 
$\pi_{M} (\widetilde{\mathcal{O}}(u_p)) = \mathcal{O}(p)$.

In order to prove that $\pi_{M}^{-1}(\mathcal{O}(p) ) \subset \widetilde{\mathcal{O}}(u_p)$ consider $v_q \in \pi_{M}^{-1}(\mathcal{O}(p))$. 
Then $q = \pi_{M} (v_q) \in \mathcal{O}(p)$ which means that $p$ and $q$ are linked by a broken curve 
$\gamma = \gamma_{n} \ast \cdots \ast \gamma_{1}$
 where each $\gamma_i$ segment is either 
unit horizontal geodesic or a reverse of a unit horizontal geodesic.


 Set   $\widetilde{\gamma}_{i}(t) :=  (\gamma_{i}'(t))_{\gamma_{i}(t)}\in T_{\gamma_{i}(t)}M$
if $\gamma_i$   a unit segment of geodesic and set $\widetilde{\gamma}_{i}(t) :=   (-\gamma_{i}'(t) )_{\gamma_{i}(t)}\in T_{\gamma_{i}(t)}M$
if $\gamma_i$ is a reverse segment of geodesic. Let $\delta_1$ be an integral line of $\mathcal{C}_1$ connecting
 $u_p$ with $\widetilde{\gamma}_{1}(0),$ $\delta_i$ be an integral line of $\mathcal{C}_1$ joining 
$\widetilde{\gamma}_{i-1}(r_{i-1})$ with $\widetilde{\gamma}_{i}(r_{i})$ and 
$\delta_{n+1}$ be an integral line of $\mathcal{C}_1$ joining $\widetilde{\gamma}_{n}(r_{n})$ with
 $v_q.$ Then 
$\delta_{n+1} \ast \widetilde{\gamma}_{n} \ast \delta_{n}\ast  \widetilde{\gamma}_{n-1} \ast \delta_{n-1} \ast \cdots \ast \widetilde{\gamma}_1 \ast \delta_1$
is a broken curve connecting $u_p$ with $v_q$, where each segment is either a integral line of $\widetilde{\mathcal{C}}$ or a reverse of a integral line of 
$\widetilde{\mathcal{C}}$ which means that $v_q \in \widetilde{\mathcal{O}}(u_p)$.

\vspace{0.25\baselineskip}

\emph{Item (d)} For a fixed $\tilde{q}_0\in N$ consider a small relatively compact trivial $\widetilde{\F}$-neighborhood $V_0$ of $\tilde{q}_0$ so that $V_{0}\cap \widetilde{\mathcal{O}}(\tilde{q}_0)$ has only one connected component. We want to check that for each $t_0$ there exists $t_{1}>t_0$ and a point $\tilde{x}\in V_{0}\cap \widetilde{\mathcal{O}}(\tilde{q}_0)$ so that $\widetilde{x}_{1} = e^{t_{1}\vec{\mathtt{g}}}(\tilde{x})\in V_{0}\cap \widetilde{\mathcal{O}}(\tilde{q}_0)$.

We claim that there exists a relatively compact  neighborhood $V_{1}\subset V_{0}$ and a neighborhood $W$ of $e\in G$ so that:
\begin{itemize}
\item if $g\in W$ then $g^{-1}\in W$ and   $\mu_{g^{-1}}(V_{1})\subset V_{0}$ and $\mu_{g}(V_{1})\subset V_{0}$;
\item if $\tilde{y}\in V_1$ then there exists $\tilde{x}\in \widetilde{\mathcal{O}}(\tilde{q}_0)\cap V_{1}$ so that $\mu_{g}(\tilde{x})=\tilde{y}$, with $g\in W$.
\end{itemize}
By Lemma \ref{lemma-poisson-stable}, for each $t_0,$ there exists $\tilde{y}\in V_{1}$ and $t_{1}>t_0$ so that $e^{t_{1}\vec{\mathtt{g}}}(\tilde{y}) = \tilde{y}_{1}\in V_{1}$. Consider $\tilde{x} \in  V_{1}\cap \widetilde{\mathcal{O}}(\tilde{q}_0)$ so that $\tilde{y}=\mu_{g}(\tilde{x})$. Since $\mu_{g}\circ e^{t_{1}\vec{\mathtt{g}}}= e^{t_{1}\vec{\mathtt{g}}}\circ\mu_{g},$ we have $\tilde{y}_{1}=e^{t_{1}\vec{\mathtt{g}}}\big(\mu_{g}(\tilde{x})\big)=\mu_{g}\big(e^{t_{1}\vec{\mathtt{g}}}(\tilde{x})\big)$ and hence $\tilde{x}_{1}:=\mu_{g^{-1}}(\tilde{y}_{1}) = e^{t_{1}\vec{\mathtt{g}}}(\tilde{x})$. Once $\mu_{g^{-1}}(V_{1})\subset V_0$ and $V_{0}\cap \widetilde{\mathcal{O}}(\tilde{q}_0)$ has only one connected component, we infer that $\tilde{x}_{1}\in V_{0}\cap \widetilde{\mathcal{O}}(\tilde{q}_0)$
as we wanted to prove.
\end{proof}

We can now end the proof of item (a) of Theorem \ref{theorem-dualizacao}. 
Item (d) of Lemma \ref{lemma-poisson-stable-na-orbita}
implies that $\vec{\mathtt{g}}$ restricts to each orbit is Poisson stable. 
From Proposition \ref{proposition-prop-814-Agrachev} 
$-\vec{\mathtt{g}}$ is compatible with $\widetilde{\mathcal{C}}$, 
i.e. $\widetilde{\mathcal{A}}_{\tilde{q}}(\widetilde{\mathcal{C}})$ is dense in 
$\widetilde{\mathcal{A}}_{\tilde{q}}(\widetilde{\mathcal{C}}\cap \{- \vec{\mathtt{g}}\})= \widetilde{\mathcal{O}}(\tilde{q})$. 
Therefore, it follows from Proposition \ref{proposition-cor-83-Agrachev} that
\begin{equation}
\label{eq-fim-prova-itema}
\widetilde{\mathcal{A}}_{\tilde{q}}(\widetilde{\mathcal{C}}) = \widetilde{\mathcal{O}}(\tilde{q}).
\end{equation}
The above equation finishes the proof as remarked before.

\vspace{\baselineskip}

\section{Proof of Proposition \ref{proposition-Orbit-positive} and item (b) of Theorem \ref{theorem-dualizacao}  }
\label{section-prof-main-proposition}

Let us start by recalling the definition of singular Finsler foliation (SFF for short),  a class of singular foliation that includes among other examples, 
 the partition of $M$ into orbits of a Finsler action;  for properties and more examples of SFF see \cite{Alexandrino-Alves-Javaloyes-Equifocal} and 
see \cite{Alexandrino-Alves-Javaloyes-SFF}.

\begin{definition}[SFF]
\label{definition-SFF}
A partition $\F=\{L\}$ on a complete 
Finsler manifold $(M,F)$ is called a 
\emph{singular Finsler foliation} if it satisfies
the following two conditions:
\begin{itemize}
\item[(a)] $\F$ is a \emph{singular foliation}, i.e., for each
$p\in M$, each basis $\{X_{i} \}$ of 
the tangent space $T_{p}L_{p}$ of the leaf $L_p$ 
through $p$, can be  extended to vector fields 
$\{ \vec{X}_{i}\} $ linearly independent, tangent
to the leaves of $\F$ near $p$. 
\item[(b)] $\F$ is  \emph{Finsler}, in other words, if a geodesic
$\gamma:(-\epsilon,\epsilon)\to M$ 
is orthogonal to the leaf $L_{\gamma(0)}$
(i.e. $g_{\gamma'(0)}(\gamma'(0),v)=0$ 
for each $v\in T_{\gamma(0)}L$) then $\gamma$
is horizontal, i.e., orthogonal to each leaf it meets.
\end{itemize}
\end{definition}

We also need to present a result that  is a direct consequence of Lemma  \ref{lemma-Wilking-decomposition}, 
a version of Wilking's lemma for Finsler geometry. 

\begin{proposition} \label{lemma-Jacobi-Decomposition}
Let $(M, F)$ be a complete Finsler manifold with non negative flag curvature along geodesic $\gamma: \mathbb{R} \to M$ orthogonal (at its initial point) to a submanifold $L \subset M$. Denote $\J_{\gamma}^{L}$ the set of all $L$-Jacobi fields along $\gamma$. Then
\begin{equation*}
\J_{\gamma}^{L} = \mathrm{\hspace{0.5ex}span\hspace{0.25ex}}_{\mathbb{R}} \{J \in \J_{\gamma}^{L} \;|\; J(t) = 0 \mathrm{\hspace{1ex} for \; some \hspace{1ex}} t \in \mathbb{R}\} \oplus \{J \in \J_{\gamma}^{L} \;|\; J \mathrm{ \hspace{1ex} is \; parallel} \}.
\end{equation*}
\end{proposition}

As explained in the introduction, item (b) of Theorem \ref{theorem-dualizacao} follows direct from item (a) of Theorem \ref{theorem-dualizacao}
and Proposition \ref{proposition-Orbit-positive}. Therefore let us prove this proposition in this section.

Let $\gamma:\mathbb{R}\to M$ be a unit speed geodesic orthogonal to a regular leaf $L$ at $\gamma(0)$ (i.e., $\gamma$ is an
horizontal geodesic). 
First we want to check that the first  summand of the decomposition presented in Proposition \ref{lemma-Jacobi-Decomposition}
is tangent to the orbit $\mathcal{O}(\gamma(0))$. To prove this it suffices to prove the next lemma


\begin{lemma}
If a Jacobi field $J\in J_{\gamma}^{L}$ has zero at $t_0$ (i.e., $J(t_0)=0$) then it is tangent to the orbit $\mathcal{O}(\gamma(0))$.
\end{lemma}
\begin{proof}

Let $t\to\gamma_{s}(t)=\gamma(s,t)$  be a  variation of horizontal unit  geodesics
orthogonal to $L_p$ with $p=\gamma(0,0)=\gamma(0)$ and so that $J(t)=\frac{\partial}{\partial s}\gamma(0,t)$.
Consider a basis $\{X_{i}\}$ of $T_{\gamma(t_0)} L_{\gamma(t_0)}.$    
It follows from item (a) of Definition \ref{definition-SFF} that these vectors can be extended to vector fields  $\{\vec{X}_i \}$
that are linearly  independent. 
From item (b) of Definition \ref{definition-SFF}, we have that $0=g_{\gamma'_{s}}(\vec{X}_{i},\gamma'_{s}).$ By Differentiating this equation 
and taking into account   item (b) of Proposition \ref{proposition-chern-connection},  we infer:
\begin{eqnarray*}
0 & = & \frac{\partial }{\partial s} g_{\gamma'_{s}}(\vec{X}_i,\gamma_{s}'(t))\\
& = &  g_{\gamma'_{s}}(\frac{\nabla^{\gamma_{s}'} }{\partial s} \vec{X}_i,\gamma_{s}'(t))\\
& + &  g_{\gamma'_{s}}(\vec{X}_i,\frac{\nabla^{\gamma_{s}'} }{\partial s} \frac{\partial }{\partial t}\gamma_{s}(t))\\
& + & C_{\gamma'_{s}}(\frac{\nabla^{\gamma_{s}'} }{\partial s}\gamma_{s}'(t), \vec{X}_{i},\gamma_{s}'(t))
\end{eqnarray*}
The above equation, the fact that 
$\frac{\nabla^{\gamma_{s}'} }{\partial s} \vec{X}_i|_{_{s=0, t=t_0}}=\nabla^{\gamma_{0}'}_{J(t_0)}  \vec{X}_i=\nabla^{\gamma_{0}'}_{0}  \vec{X}_i=0$ 
and Eq.\eqref{eq-propriedade-tensor-Cartan}
allow us to  conclude that 
\begin{equation}
\label{eq-prova-prop-orbit-positive-derivadaJ}
 0=  g_{\gamma'_{0}}\Big(\vec{X}_i,\frac{\nabla^{\gamma_{0}'} }{\partial t} \frac{\partial }{\partial s}\gamma_{s}(t_{0} )|_{s=0}\Big)
= g_{\gamma'_{0}}(X_i,\frac{\nabla^{\gamma_{0}'} }{\partial t} J(t_{0})).
\end{equation} 

Eq.\eqref{eq-prova-prop-orbit-positive-derivadaJ} implies that 
\begin{equation} 
\label{eq2-prova-prop-orbit-positive-derivadaJ}
J'(t_0)\in  \mathcal{H}(t_0) \mathrm{\,  and \,} J(t_0)=0.  
\end{equation}
Here $\mathcal{H}(t_0)=\{w \in T_{\gamma(t_0)}M | , g_{\gamma'(t_0)}(w, X_{i}(\gamma(t_0))=0, \forall i\}.$

We claim that  \emph{the normal cone $\nu_{\gamma(t_0)}(L_{\gamma(t_0)})$ is tangent to $\mathcal{H}(t_0)$}.
In order to check this claim, consider a curve  $s\to v(s)$ with $v(0)=\gamma'(t_0)$ contained in 
 the unit normal cone, i.e., $g_{v(s)}(v(s), X_i(\gamma(t_0)))=0,\, \forall i.$ By differentiating this equation with respect to $s$
and  taking into account  item (b) of Proposition \ref{proposition-chern-connection}, we infer  
 that $g_{\gamma'(t_0)}(v'(0),X_{i}(\gamma(t_0)))=0,  \, \forall i$, i.e, that $v'(0)$ is tangent to $\mathcal{H}(t_0)$. 
An argument comparing dimensions allows one to conclude
the proof of the claim, see also proof of \cite[Lemma 2.9]{Alexandrino-Alves-Javaloyes-Equifocal}.

The claim and  Eq.\eqref{eq2-prova-prop-orbit-positive-derivadaJ} imply
the existence of a variation of geodesics $t\to f(s,t)$ so that
\begin{itemize}
\item $t\to f(0,t)=\gamma(t),$
\item  $t\to f(s,t)$  are geodesics orthogonal to $L_{\gamma(t_0)}$ i.e., contained in $\mathcal{O}(\gamma(t_0)),$
\item $f(s,t_0)=f(0,t_0)=\gamma(t_0)$ and $J(t)=\frac{\partial}{\partial s} f(0,t).$
\end{itemize}
In fact we can define $f(s,t)=\pi\Big( e^{(t-t_0)\vec{\mathtt{g}}}v(s) \Big)$
 where $\pi:TM\to M$ is the canonical projection and   $s\to v(s)$  is a curve contained 
 in unit cone $\nu_{\gamma(t_0)}^{1}(L_{\gamma(t_0)})$ with $v'(0)=J'(t_0)$ and $v(0)=\gamma'(t_0).$ 

Since $t\to f(s,t)$ are geodesics  contained in $\mathcal{O}(\gamma(t_0))$ and  
 $\gamma(0)\in \mathcal{O}(\gamma(t_{0})),$ we conclude that the variation  $t\to f(s,t)$ is contained in 
$\mathcal{O}(\gamma(0))$ and hence that $t\to J(t)=\frac{\partial}{\partial s} f(0,t)$ is tangent  to  $\mathcal{O}(\gamma(0))$ what finishes the proof.

\end{proof}

Let us now check that codimension of the dual leaf is zero.
Assume by contradiction that there exists $v \in L_{\gamma(0)}$ orthogonal to $\mathcal{O}(\gamma(0))$, 
where $\gamma$ is an horizontal geodesic with $\gamma(0)=q$ with $K(q) > 0$ for $q\in L_{q_0}$. 
Consider a $L_{\gamma(0)}$-Jacobi field $J$, so that $J(0)=v$. 
Since we have proved above that the first  summand of the decomposition presented in Proposition \ref{lemma-Jacobi-Decomposition}
is tangent to the orbit $\mathcal{O}(\gamma(0))$, we have that $J$ can not be contained in this summand. Hence, by Proposition \ref{lemma-Jacobi-Decomposition}, 
$J$ must be a non trivial parallel Jacobi vector field, that implies that
the curvature can not be positive, what contradicts our hypothesis that $K(q)>0$.

Since we have proved that $\mathcal{O}(q)$ has codimension zero for each $q\in L_{q_0}$ and each point $x\in M$ is contained in
an orbit $\mathcal{O}(q)$ (for $q\in L_{q_0}$), we conclude that $\mathcal{O}(q_0)=M$.

It follows from item (a) of Theorem \ref{theorem-dualizacao} that $M=\mathcal{O}(q_0)=\mathcal{O}(q)=\mathcal{A}_{q}(\mathcal{C})$. 
This conclude the proof of Proposition \ref{proposition-Orbit-positive} and hence the proof 
 of item (b) of Theorem \ref{theorem-dualizacao}.


\vspace{\baselineskip}

\section{Wilking's transverse Jacobi fields} \label{Wilking-transverse-Jacobi-fields}

We reproduce here, in more general context, the construction of transverse Jacobi fields presented in \cite{Wilking-duality}, in \cite{Lytchak}
and in \cite{lecturesRadeschi} in order to obtain a Fislerian version of the Corollary 10 in \cite{Wilking-duality}, 
i.e., Proposition \ref{lemma-Jacobi-Decomposition}.

\vspace{0.25\baselineskip}

\subsection{Jacobi Triples and Jacobi Equation} \label{Jacobi-triples-Jacobi-equation}

A \emph{Jacobi triple} $(E, D, R)$ is composed by
\begin{itemize}
\item Euclidean vector field $E$ (total space) over a open interval $I \subset \mathbb{R}$ (with rank $n$);

\item a covariant derivative $D: \Gamma(E) \to \Gamma(E)$ compatible with the fiberwise metric of $E$;

\item a self-adjoint $C^\infty (I)$-homomorphism $R: \Gamma(E) \to \Gamma(E)$.
\end{itemize}

Presented in this way, this definition seems a little artificial. Indeed, this is an algebraic approach 
 whose intention is to condense the relevant data and properties of Jacobi fields that will be useful throughout section \ref{Wilking-transverse-Jacobi-fields}.

Given a Jacobi triple $(E, D, R)$, the kernel of the second order differential operator $D^2 + R$ will be called the space of $(E, D, R)$-Jacobi fields (or simply \emph{Jacobi fields}) and will be denoted by $\J (E, D, R)$ (or simply by $\J$). Since $D^2 + R$ is a linear operator, for each $t \in I$, the map $J \longmapsto (J(t), DJ(t))$ is an isomorphism between $\J$ and $E_t \oplus E_t$. In particular, $\dim (\J) = 2 \rank (E) = 2 n$.

The space of Jacobi fields $\J$ inherit a canonical symplectic form $\omega$ given by
\begin{equation*}
\omega (J_1, J_2) := \Scal{D J_1}{J_2} - \Scal{J_1}{D J_2}.
\end{equation*}
Note that the right term of this definition is in fact independent of the $t$ parameter. Precisely $\omega$ is the pullback of the canonical form by the isomorphism $J \longmapsto (J(t), DJ(t))$. As usual, the space of Lagrangian subspaces of $(\J, \omega)$ (the Grassmannian Lagrangian of $(\J, \omega)$) will be denoted by $\Lambda (\J)$, i.e.
\begin{equation} \label{Lagrangian-Grassmannian}
\Lambda (\J) := \{\L \subset \J : \L \text{ is a Lagrangian subspace of } \J\}.
\end{equation}

Finally we establish some useful notation. Given $\I \subset \J$ a vector subspace of Jacobi fields for each $t \in I$ it will be denoted
\begin{equation*}
\I (t) := \{J(t) \in E_t : J \in \I\} \qquad\text{and}\qquad \I^{0}_{t} := \{J \in \I : J(t) = 0\}.
\end{equation*}
Following this notation we point out that $\I^{0}_{t}$ is isomorphic to $D (\I^{0}_{t}) (t) = \{D J (t) : J \in \I^{0}_{t}\}$.

\vspace{0.25\baselineskip}

\subsection{Illustrative examples} \label{illustrative-examples}

Looking for a consolidation of our algebraic approach we present some geometrical examples of subspaces of Jacobi fields in an increasing rate of complexity. The focus is the subspaces determined by the symplectic structure (i.e. isotropic and Lagrangian subspaces). Some future notation will be presented as well.

\vspace{0.15\baselineskip}
\begin{example}[Finslerian Jacobi fields]
Let $M$ be a Finsler manifold with fundamental tensor $g$ and $\gamma$ a geodesic segment. Observe that $(\gamma^\ast TM, g_{\gamma'(\cdot)})$ is a Euclidian vector bundle over $I$. Denote by $D_{\gamma}^{\gamma'}$ the Chern covariant derivative along $\gamma$ and $R^{\gamma'}$ the Jacobi operator along $\gamma$. Then $(\gamma^\ast TM, D^{\gamma'}, R_{\gamma'})$ is a Jacobi triple.

The set of Jacobi fields associated with $(\gamma^\ast TM, D^{\gamma'}, R_{\gamma'})$ will be denoted by $\J_\gamma$.
\end{example}
\vspace{0.15\baselineskip}

In the next two examples we are going to use the notation established in the previous example.

\vspace{0.15\baselineskip}
\begin{example}[L-Jacobi fields] \label{L-Jacobi-fields-I}
Let $M^n$ be a Finsler manifold, $L \subset M$ a immersed submanifold and $\gamma: [a, b] \subset \R \longrightarrow M$ a geodesic segment such that $\gamma(a) \in L$ and $\gamma'(a)$ is $g_{\gamma'(a)}$-orthogonal to $L$.
Denote $\text{pr}_L: T_{\gamma(a)} M \to T_{\gamma(a)}L$ the canonical projection with respect to the $g_{\gamma'}$-orthogonal decomposition of $T_{\gamma(a)} M$.

The set o $L$-Jacobi fields is defined by
\begin{equation*}
\J_{\gamma}^{L} := \{J \in \J_\gamma : J(a) \in T_{\gamma(a)} L \text{ \hspace{0.25ex}and\hspace{0.25ex} } \text{pr}_L (D_{\gamma}^{\gamma'} J (a)) = S_{\gamma'(a)} (J (a))\},
\end{equation*}
where $S_{\gamma'(a)}$ is the shape operator of $L$ in the direction $\gamma'(a)$.
As we previously discuss, this is precisely the Jacobi fields obtained by variations of $\gamma$ through geodesics starting perpendicular to $L$.
An advantage in this presentation is that will became more easily to see that $\J_{\gamma}^{L}$ is a Lagrangian subspace of Jacobi fields. Clearly the self-adjointness of $S_{\gamma'(a)}$ guarantees that $\J_{\gamma}^{L}$ is isotropic, i.e, 
\begin{align*}
\omega (J_1, J_2) &= \langle D_{\gamma}^{\gamma'} (J_1), J_2 \rangle - \langle J_1, D_{\gamma}^{\gamma'} (J_2) \rangle \\
&= \langle - S_{\gamma'(a)} (J_1 (a)), J_2 (a) \rangle - \langle J_1 (a), - S_{\gamma'(a)} (J_2 (a)) \rangle \\
&= 0
\end{align*}
where $J_1, J_2 \in \J_{\gamma}^{L}$. The dimension of $\J_{\gamma}^{L}$ is determined by the linearly independent choices for the initial conditions of its Jacobi fields which implies that $\dim (\J_{\gamma}^{L}) = \dim (T_{\gamma(a)} L) + \dim ((T_{\gamma(a)} L)^\omega) = \frac{1}{2} \dim (\J_\gamma)$.
\end{example}

\vspace{0.15\baselineskip}

\begin{example}[Finsler submersions] \label{Finsler-Submersions-Example}
Let $\pi: M^{m+k} \longrightarrow B^{k}$ be a Finsler submersion and $\gamma: [a, b] \subset \R \longrightarrow M$ a horizontal geodesic;
 see Figure \ref{figura1-Wilking-distribution}. Along $\gamma$ it is possible to consider a horizontal bundle $H$ by $g_{\gamma'}$-orthogonal complement of the vertical bundle $V := \gamma^\ast \Ker (d \pi)$, which will allow us to define operators analogous to the O'Neill tensors in Riemannian submersions and thus be able to work with holonomy type of Jacobi fields and projectable Jacobi fields just like in the Riemannian case. More precisely, we  consider $S_{\gamma'}: \Gamma(V) \to \Gamma(V)$ the shape operator of the fibers in the direction $\gamma'$, we  define $\mathbb{A}_{\gamma'}: \Gamma(\gamma^\ast TM) \to \Gamma(\gamma^\ast TM)$ by $\mathbb{A}_{\gamma'} (X) := (D_{\gamma}^{\gamma'} X^V)^H + (D_{\gamma}^{\gamma'} X^H)^V$ and denote $A_{\gamma'}:= \mathbb{A}_{\gamma'} |_{\Gamma(H)}$.

\vspace{0.15\baselineskip}

The set of holonomy type Jacobi fields (along $\gamma$) is defined by
\begin{equation*}
\J_{\gamma}^{\hol} := \{J \in \J_\gamma : J (a) \in V_{\gamma(a)} \text{ \hspace{0.25ex}and\hspace{0.25ex} } D_{\gamma}^{\gamma'} J (a) = - (S_{\gamma'(a)} + A_{\gamma'(a)}^{\ast}) (J(a))\}
\end{equation*}
and this is an example of isotropic subspace of $\J_\gamma$. This is simple to verify since the shape operator is self-adjoint and the holonomy type Jacobi fields are vertical in its initial point. In fact they are everywhere vertical. Additionally, once a Jacobi field is determined by its initial conditions it is simple to conclude that $\dim (\J_{\gamma}^{\hol}) = \rank (V)=k$.

Like in the case of $L$-Jacobi fields it is more suitable for computations to define $\J_{\gamma}^{\hol}$ algebraically and postpone its geometrical meaning.
 In this case they can be obtained by variations of $\gamma$ through horizontal geodesics that are horizontal lifts of the geodesic $\pi \circ \gamma$.

\vspace{0.15\baselineskip}

Another remarkable space of Jacobi field associated with a Finsler 
submersion is the set of projectable Jacobi fields which is defined by
\begin{equation*}
\J_{\gamma}^{\proj} := \{J \in \J_\gamma : D_{\gamma}^{\gamma'} J^V = - S_{\gamma'} (J^V) - A_{\gamma'} (J^H)\}
\end{equation*}
and this is an example of coisotropic subspace of $\J_\gamma$. More precisely this is the symplectic orthogonal of $\J_{\gamma}^{\hol}$. To see this, we first observe that $\J_{\gamma}^{\proj} \subset (\J_{\gamma}^{\hol})^\omega$ which is a directly consequence of the subsequent computation 
\begin{align*}
\omega (J_1, J_2) &= \langle D_{\gamma}^{\gamma'} J_1, J_2 \rangle - \langle J_1, D_{\gamma}^{\gamma'} J_2 \rangle \\
&= \langle - S_{\gamma'(a)} (J_1 (a)), J_{2}^{V} (a) \rangle + \langle -A_{\gamma'(a)}^{\ast}(J_1 (a)), J_{2}^{H} (a) \rangle + \\
&\quad - \langle J_1 (a), -S_{\gamma'(a)} (J_{2}^{V} (a)) \rangle - \langle J_1 (a), -A_{\gamma'(a)} (J_{2}^{H}(a)) \rangle \\
&= 0
\end{align*}
where $J_1 \in \J_{\gamma}^{\hol}$ and $J_2 \in \J_{\gamma}^{\proj}$. Since the dimension of a subspace of Jacobi fields is determined by linearly independent choices for the initial conditions, clearly we have $\dim (\J_{\gamma}^{\proj}) = \dim (T_{\gamma(a)} M) + \rank (H) = 2n + k$. 

It is worth to mention that the projectable Jacobi fields can be obtained by variations of $\gamma$ through horizontal geodesics. The term ``projectable'' is due to the fact that each Jacobi field of $\J_{\gamma}^{\proj}$ is $\pi$-related to a Jacobi field of $\J_{\pi \circ \gamma}$. In fact, there is a well defined projection map $\pi_\ast: \J_{\gamma}^{\proj} \to \J_{\pi \circ \gamma}$ which is surjective and such that $\Ker (\pi_\ast) = \J_{\gamma}^{\hol}$.

\vspace{0.25\baselineskip}

In conclusion, a Finsler submersion offers to us an example where all the symplectic-types subspaces of Jacobi fields (isotropic $\I$, Lagrangian $\L$ and coisotropic $\I^\omega$) are present, as follows
\begin{center}
\vspace{0.5\baselineskip}
\begin{tabular}{ccccc}
$\I$&$\subset$&$\L$&$\subset$&$\I^\omega$ \\
\rotatebox[origin=c]{90}{$=$}&&\rotatebox[origin=c]{90}{$=$}&&\rotatebox[origin=c]{90}{$=$} \\
\\[-2ex]
$\J_{\gamma}^{\hol}$&$\subset$&$\J_{\gamma}^{M_b}$&$\subset$&$\J_{\gamma}^{\proj}$
\end{tabular}
\vspace{0.5\baselineskip}
\end{center}
where $M_b$ is the fiber of $\pi$ through $b = \pi(\gamma(a))$. Furthermore, as consequence of the isomorphism theorem applied to $\pi_\ast$, we have the following nice geometric interpretation for the symplectic reduction
\begin{equation*}
\frac{\I^\omega}{\I} = \frac{\J_{\gamma}^{\proj}}{\J_{\gamma}^{\hol}} \approx \J_{\pi \circ \gamma}.
\end{equation*}
\end{example}

\vspace{0.25\baselineskip}

\begin{center}
\begin{figure}[tbp]
	\includegraphics[scale=0.3]{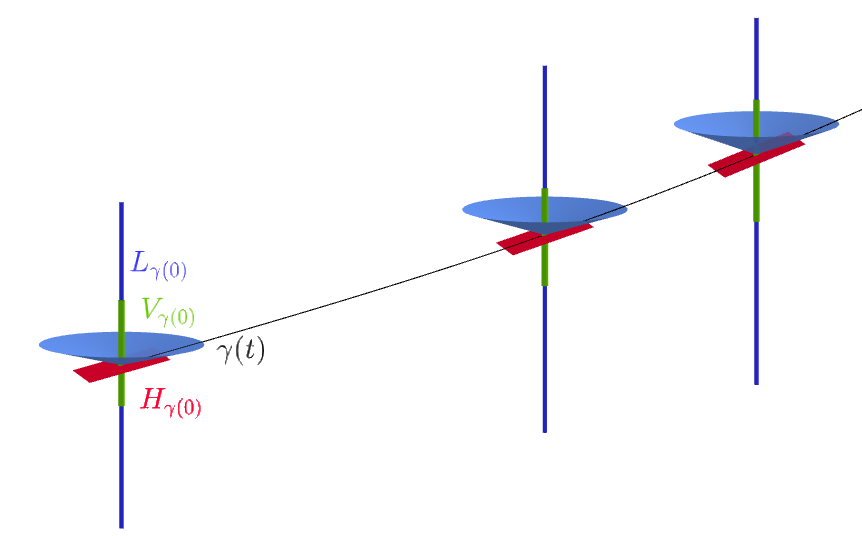}
	\caption{Figure generated by the 
software geogebra.org  illustrating a submersion on  Randers space, 
with horizontal geodesic $\gamma$, horizontal bundle  \textcolor{red}{$H_t$} and 
the vertical bundle \textcolor{Green}{$V_t$} of  Finsler submersion, see Example \ref{Finsler-Submersions-Example}}
	\label{figura1-Wilking-distribution}
\end{figure}
\end{center}

\subsection{Structures associated with isotropic subpaces of Jacobi fields} \label{Structures-associated-isotropic-subpaces-Jacobi-fields}

We stress in this section some technical results required to prove Lemma \ref{Wilking-decomposition-lemma} in a manner that the reader can skip the proofs in this section without a further damage in the comprehension of this lemma and its proof.

More precisely we are going to generalize some aspects of the holonomy type Jacobi fields present in Example \ref{Finsler-Submersions-Example} to any isotropic subspace $\I$ of a given Jacobi triple $(E, D, R)$. Two structures associated with a $\I$ will be explored, the $\I$-transverse Jacobi fields and the $\I$-Riccati operators. Although this generalization seems futile it will be quite useful in the proof of Lemma \ref{Wilking-decomposition-lemma} where the choice of a specific isotropic subspace $\I$ is the key idea of the proof.

\subsubsection{Horizontal bundle and transverse Jacobi fields} \label{Horizontal-bundle-transverse-Jacobi-fields}

We start with a Lemma that describes the vertical bundle associated with an isotropic subspace $\I$.

\vspace{0.15\baselineskip}
\begin{lemma} \label{properties-isotropic}
Let $\I \subset \J$ be an isotropic subspace. Then
\begin{enumerate}
\item \label{properties-isotropic-c}
{\bf(Singular instants)} The set $\{t \in I \;|\; \I_{t}^{0} \neq \{0\}\} \subset I$ is discrete.
\item \label{properties-isotropic-d}
{\bf(Vertical bundle)} The set
\begin{equation*}
V^\I := \coprod_{t \in I} \, \I(t) \oplus D(\I_{t}^{0})(t)
\end{equation*}
is a vector subbundle of $E$ with rank equal to the dimension of $\I$.
\end{enumerate}
\end{lemma}

\begin{proof} \quad
The proof is similar to the proof presented in Lemma 3.3 of \cite{Alexandrino-Alves-Javaloyes-Equifocal}.
\end{proof}
\vspace{0.25\baselineskip}

Following the previous lemma, given an isotropic subspace $\I$, the set
\begin{equation} \label{regular-instants}
I_\I := \{t \in I \;|\; \I_{t}^{0} = \{0\}\}
\end{equation}
will be called the set of $\I$-\emph{regular instants} and in a logical contrast his complement will be called the set of $\I$-\emph{singular instants} (item \ref{properties-isotropic-c}).

Furthermore the item \ref{properties-isotropic-d} of the lemma associates $\I$ to a vector subbundle $V^{\I} \subset E$ which will be called the \emph{vertical subbundle} (associated with $\I$) and its orthogonal complement $H^{\I}$ will be called \emph{horizontal subbundle} (associated with $\I$). The names of this subbundles are inspired by Example \ref{Finsler-Submersions-Example}.
When these is no risk of confusion, the spaces $V^{\I}$ and $H^{\I}$  will be simply  denoted by $V$ and $H$, respectively.

It is immediate from the item \ref{properties-isotropic-d} that for each $t \in I_\I$ we have $\I(t) = V_t$. In particular for any Lagrangian subspace $\L $ we have $\L(t) = E_t$ for all $t \in I_\L$ by dimension issues.

Relatively to the decomposition $E = V \oplus H$ we denote the horizontal projection by $\pi_H$ and the horizontal components of the operators $D$ and $R$ by $D_H$ and $R_H:= (R|_{H})^{H}$ i.,e the $H$-component of the restriction of $R$
to $H$. To deal with mixed components of the covariant derivative we define the tensor $\mathbb{A}: \Gamma(E) \to \Gamma(E)$ as
\begin{equation*}
\mathbb{A} (X) := (D X^V)^H + (D X^H)^V;
\end{equation*}
c.f. the definition of  O'Neill tensor in Example \ref{Finsler-Submersions-Example}.

It is straightforward to see that $\mathbb{A}$ is in fact a $C^\infty(I)$-homomorphism. Also $\mathbb{A}$ has the remarkable property that $\mathbb{A}_V = - \mathbb{A}_{H}^{\ast}$ which is proved in \cite{Alexandrino-Alves-Javaloyes-Equifocal} and as a consequence 
\begin{equation}
\label{eq-propriedades-tensor-A-Aadjunto}
\mathbb{A}^2 |_H = \mathbb{A}_V \mathbb{A}_H = - \mathbb{A}_{H}^{\ast} \mathbb{A}_H 
\end{equation}
is self-adjoint nonpositive operator.

\begin{proposition}[transverse Jacobi equation] \label{transverse-Jacobi-fields}
Let $\I \subset \J$ be an isotropic subspace. Then
\begin{enumerate}
\item \label{transverse-Jacobi-equation-a}
$(H, D_H, R_H - 3 \mathbb{A}^2 |_{H})$ is a Jacobi triple. 
\item \label{transverse-Jacobi-equation-b}
$\Ker (\pi_{\mathrm{H}} |_{\I^\omega}) = \I$ and $\pi_{\mathrm{H}} (\I^{\omega}) = \J_{\I}$,
where $\J_{\I}$ is the space of Jacobi fields associated to $(H, D_H, R_H - 3 \mathbb{A}^2 |_{H})$. 
\item \label{transverse-Jacobi-equation-c}
$\pi_{\mathrm{H}}: \I^\omega / \I \longrightarrow \J_{\I}$ is a symplectic isomorphism.
\end{enumerate}
\end{proposition}
\begin{proof} The item \ref{transverse-Jacobi-equation-a} is immediate, once $\mathbb{A}^2 |_{H}$ is self-adjoint. Then we proceed with the proof of itens \ref{transverse-Jacobi-equation-b} and \ref{transverse-Jacobi-equation-c}.
\begin{enumerate}
\item[(b)] First we are going to prove that $\Ker (\pi_{\mathrm{H}} |_{\I^\omega}) = \I$. It is easily to check that $\Ker (\pi_{\mathrm{H}} |_{\I^\omega}) \supset \I$ so we are going to concentrate in prove that $\Ker (\pi_{\mathrm{H}} |_{\I^\omega}) \subset \I$. Given $[J] \in \Ker (\pi_{\mathrm{H}} |_{\I^\omega})$ the vector subspace $\widehat{\I} = \I + \mathbb{R} J$ is isotropic (since $J\in \I^\omega$ )
 and for any $\widehat{\I}, \I$-regular $t \in I$
\begin{equation}
\label{eq-J-em-I}
\dim (\I + \mathbb{R} J)  \stackrel{(*)}{=} \dim (\widehat{\I} (t)) \stackrel{(**)}{=} \dim (\I (t)) \stackrel{(*)}{=} \dim (\I)
\end{equation}
where the  equality ($\ast $) follows from the fact that $t$ is regular and the equality ($\ast \ast$) follows from the fact that
$J\in \mathrm{Ker} (\pi_{H})$, i.e, $J(t)$ is vertical. Eq. \eqref{eq-J-em-I}  implies that $J \in \I$.

\vspace{0.25\baselineskip}

Now we are going to prove that $\pi_{\mathrm{H}} (\I^{\omega}) = \J_{\I}$. 
Note that it suffices to prove that $\pi_{\mathrm{H}} (\I^{\omega}) \subset \J_{\I}$ since
\begin{equation*}
\rank (\pi_{\mathrm{H}} |_{\I^\omega}) = \dim (\I^\omega) - \dim (\Ker (\pi_{\mathrm{H}} |_{\I^\omega})) = 2 \dim (H).
\end{equation*}
Also note that:

\vspace{\baselineskip}
\textbf{Claim} Given $\widehat{J}\in \I^{\omega}$ there exists
$J\in \I^{\omega}$ so that
\begin{enumerate}
\item[(1)] $J(t_0)= \widehat{J}^{H}(t_0)$, where $t_0$ is a $\I$-regular  time. 
\item[(2)] $[J]=[\widehat{J}]$, i.e., $\pi_{H}(J)=\pi_{H}(\widehat{J}).$
\end{enumerate}
\vspace{0.25\baselineskip}
In fact, since $t_0$ is $\I$-regular, we have
$\I (t_0)=  V_{t_0}$ and
there exists $\widetilde{J}\in \I$
such that  $\widetilde{J}(t_0)=\widehat{J}^{V}(t_0).$ 
Set $J:=\widehat{J}-\widetilde{J}\in \I^{\omega}$. 
Clearly $[J]=[\widehat{J}]$ and $J(t_0)=\widehat{J}(t_0)-\widetilde{J}(t_0)=\widehat{J}^{H}(t_0)$,
and this conclude the proof of the claim.

\

\vspace{0.25\baselineskip}

Fix a $\I$-regular $t_0 \in I$, $u_{t_0} \in H_{t_0}$, $J \in \I^\omega$ such that $J (t_0) \in H_{t_0}$, 
$X \in \Gamma (H)$ $D_H$-parallel such that $X(t_0) = u_{t_0}.$ 
Due to the above claim, in order to prove that $\pi_{\mathrm{H}} (\I^{\omega}) \subset \J_{\I}$ it suffices 
to prove  Eq.\eqref{eq-goal-jacobi-transverse-Marcelo} below.   
\begin{equation}
\label{eq-goal-jacobi-transverse-Marcelo}
\Scal{D_{H}^{2} J^H(t_0) }{u_{t_0}}= - \Scal{(R_H - 3 \mathbb{A}^2 |_{H}) J (t_0)}{u_{t_0}}. 
\end{equation}

Let us accept for a moment the following two equations that we are going  check later:
\begin{equation}
\label{eq-1-auxiliar-jacobi-transverse-Marcelo}
\langle J(t_0), D^{2}X(t_0)\rangle= 
\langle J(t_0), \mathbb{A}^{2}(X)(t_0)\rangle.
\end{equation}
\begin{equation}
\label{eq-2-auxiliar-jacobi-transverse-Marcelo}
\langle (DJ)^{V}(t_0), v\rangle = 
\langle \mathbb{A}^{*}(J^{H}(t_0)),v\rangle \, \forall v\in V_{t_0}.
\end{equation}

Replacing Eq.\eqref{eq-1-auxiliar-jacobi-transverse-Marcelo} and 
Eq.\eqref{eq-2-auxiliar-jacobi-transverse-Marcelo} in the equation below (evaluated at $t=t_0$)
 we conclude the desired Eq.\eqref{eq-goal-jacobi-transverse-Marcelo}. 
\begin{align*}
\Scal{D_{H}^{2} J^H }{X}&  = \Scal{J^H}{X}'' \\
         & = \Scal{J}{X}''\\
 &  = \Scal{D^2 J }{X} + 2 \Scal{D J}{D X} + \Scal{J}{D^2X } \\
& = \Scal{-R J }{X} + 2 \Scal{D J}{(D X)^{V}} + \Scal{J}{D^2X }\\
& = \Scal{(-R J)^{H} }{X} + 2 \Scal{(D J)^{V}}{(\mathbb{A}_{H} X)} + \Scal{J}{D^2X }.
\end{align*}

We now check Eq.\eqref{eq-1-auxiliar-jacobi-transverse-Marcelo}. 
\begin{align*}
 \langle J(t_0),D^{2}X(t_0)\rangle & = 
 \langle J(t_0),(D^{2}X)^{H}(t_0)\rangle\\
& = \langle J(t_0),(D(DX)^{V})^{H}(t_0)\rangle\\
& = \langle J(t_0),\mathbb{A}(DX)^{V}(t_0) \rangle\\
& = \langle J(t_0),\mathbb{A}\mathbb{A} X (t_0) \rangle
\end{align*}


Finally we  check Eq.\eqref{eq-2-auxiliar-jacobi-transverse-Marcelo}. Consider 
$\widetilde{J}\in \I$ so that $\widetilde{J}(t_0)=v\in V_{t_0}$
\begin{align*}
\langle (DJ)^{V}(t_0),  \widetilde{J} (t_0)\rangle
& =  \langle DJ(t_{0}), \widetilde{J}(t_0) \rangle\\
& =  \langle J(t_0), (D\widetilde{J})^{H}(t_0)\rangle \\
& =  \langle J(t_0), \mathbb{A}\widetilde{J}(t_0)\rangle\\
& =  \langle \mathbb{A}^{*} J(t_0), \widetilde{J}(t_0)\rangle\\
& =  \langle \mathbb{A}^{*} J(t_0)^{H}, \widetilde{J}(t_0) \rangle.
\end{align*}


\vspace{0.5\baselineskip}

\item[(c)] Fix a $\I$-regular $t_0 \in I$, $[J_1], [J_2] \in \I^\omega$ such that $J_1 (t_0), J_2 (t_0) \in H_{t_0}$ and $X_1, X_2 \in \Gamma (H)$ $D_H$-parallel such that $X_i (t_0) = J_i (t_0)$ for $i = 1, 2$. Then
\begin{align*}
\omega ([J_1], [J_2]) &= \Scal{D J_1 (t_0)}{X_2 (t_0)} - \Scal{X_1 (t_0)}{D J_2 (t_0)} \\
&= \left( \Scal{J_{1}^{H}}{X_2} - \Scal{X_1}{J_{2}^{H}} \right)'(t_0) \\
&= \Scal{D_H J_{1}^{H} (t_0)}{X_2 (t_0)} - \Scal{X_1 (t_0)}{D_H J_{2}^{H} (t_0)} \\
&= \omega_{\I} (J_{1}^{H}, J_{2}^{H}),
\end{align*}
where $\omega_{\I}$ is the symplectic form associated to $(H, D_H, R_H - 3 \mathbb{A}^2 |_{H})$.
\end{enumerate}
\end{proof}

The space $\J_\I$ of Jacobi fields associated with the Jacobi triple $(H, D_H, R_H - 3 \mathbb{A}^2 |_{H})$ will be called the 
\emph{space of $\I$-transverse Jacobi fields.} As well as $\J$, the set of transverse Jacobi fields possess a symplectic form $\omega_\I$ (see Section \ref{Jacobi-triples-Jacobi-equation}). It is quite usefull to mention that the Lagrangian subspaces of $(\J_\I, \omega_\I)$ has a nice description which relates them to the Lagrangian subspaces of $\J$. This is the content of the subsequent corollary which follows as a consequence of the item \ref{transverse-Jacobi-equation-c} of the previous proposition.

\begin{corollary}[transverse Lagrangian subspaces] \label{transverse-Lagrangian-subspaces}
The map
\begin{equation*}
\{\L \in \Lambda (\J) \;|\; \I \subset \L\} \ni \L \longrightarrow \pi_H (\L / \I) \in \Lambda (\J_{\I})
\end{equation*}
is a bijection.
\end{corollary}

\begin{remark}
A geometric view of the transverse Jacobi vector fields, in a particular case, could be drawn from Example \ref{Finsler-Submersions-Example}. Given a Finsler submersion $\pi: M \to B$ and a horizontal geodesic $\gamma$ it was presented in that example a isomorphism between the symplectic reduction $\J_{\gamma}^{\proj} / \J_{\gamma}^{\hol}$ (quotient between projectable Jacobi fields and holonomy Jacobi fields) and the space $\J_{\pi \circ \gamma}$ of Jacobi fields along the projected geodesic. Then from item \ref{transverse-Jacobi-equation-c} of the previous proposition, the space of $\J_{\gamma}^{\hol}$-transverse Jacobi fields is isomorphic to the $\J_{\pi \circ \gamma}$.
\end{remark}

\vspace{0.25\baselineskip}

\subsubsection{Riccati operators} \label{Riccati-operators}

Let $\L \subset \J$ be a Lagrangian subspace. Then for each $t \in I_\L$ the linear operator $S_{t}^{\L} : E_t \to E_t$ given by $S_{t}^{\L} (u_t) := D J(t)$, where $J \in \L$ is such that $J(t) = u_t$, is well defined and self-adjoint, see \cite{Gromoll-Walschap}. 
Therefore it induces a $C^\infty (I_\L)$-endomorphism in $\Gamma (E_{I_\L})$, the \emph{Riccati operator} (associated with $\L$), which will be denoted by $S^\L$.

\vspace{0.15\baselineskip}
\begin{lemma} \label{Lagrangian-spaces-Riccati-operators}. Assume that $I_\L = \R$.
Let $\xi \in \mathbb{O}(E)$ be a $D$-parallel orthonormal frame. Then
\begin{enumerate}
\item \label{Lagrangian-spaces-Riccati-operators-a}
$[S^\L]_\xi$ is a solution of the Riccati differential equation (in the space of symmetric matrices $\mathbb{M}_{n}^{\sym} (\R)$)
\begin{equation} \label{Riccati-equation}
X' + X^2 + [R]_\xi = 0.
\end{equation}
Moreover $\Ker (S^\L-D) = \L$.

\vspace{0.25\baselineskip}

\item \label{Lagrangian-spaces-Riccati-operators-b}
Given $X: I_X \subset I \to \mathbb{M}_{n}^{\mathrm{\hspace{0.25ex}sym\hspace{0.1ex}}} (\mathbb{R})$ a solution of Eq. \ref{Riccati-equation} 
and denoting by $S^X$ the $C^\infty (I_X)$-endomorphism in $\Gamma (E |_{I_X})$ characterized by $[S^X]_\xi = X$, the subspace $\L := \Ker (S^X - D) \subset \J$ is Lagrangian and $S^\L = S^X$.
\end{enumerate}
\end{lemma}

\begin{proof} \quad
\begin{enumerate}
\item[(a)] It is immediate to see that $S^{L} = D$ in $\L$ from which follows that
\begin{equation*}
- [R]_\xi [J]_\xi = [D^2 J]_\xi = [D (S J)]_\xi = [S]_{\xi}^{'} [J]_\xi + [S]_\xi [DJ]_\xi.
\end{equation*}
Then the fact that $\L(t) = E_t$ for all $t \in I_\L$ and the item \ref{properties-isotropic-c} of Lemma \ref{properties-isotropic} concludes the proof of this item.

\vspace{0.25\baselineskip}

\item[(b)] Since $X$ is symmetric we have that $S^X$ is self-adjoint so $\Ker (S^X - D)$ is isotropic. Furthermore $\dim (\Ker (S^X - D)) = \rank (E)$ (since $S^X -D$ is a differential operator of order $1$). Then  $\Ker (S^X - D)$ is in fact Lagrangian.
\end{enumerate}
\end{proof}

\begin{corollary} \label{Riccati-Equation-Real}. Assume that $I_\L = \R$. 
The function $I_\L \ni t \to \tr (S_{t}^{\L}) \in \R$ is a solution for the following Riccati equation
\begin{equation*}
x' + x^2 + r = 0
\end{equation*}
where $r: I_\L \subset \R \to \R$ is given by $r = \tr (R) + (\tr ({S^\L}^2) - \tr (S^\L)^2)$.
\end{corollary}
\vspace{0.25\baselineskip}

The previous lemma creates, for a fixed $D$-parallel orthonormal frame $\xi$, a bijection between the Lagrangian Grassmannian $\Lambda (\J)$ (see Eq. \ref{Lagrangian-Grassmannian}) and the space of solutions of the Riccati differential equation $X' + X^2 + [R]_\xi$ in the space of real symmetric matrices.

Various comparison result related to this type of Riccati differential equation was presented in  \cite{Eschenburg-Heintze}. Here we states a more weak result which will be useful for the proof of Wilking's decomposition lemma (Lemma \ref{lemma-Wilking-decomposition}).

\vspace{0.15\baselineskip}
\begin{proposition} \label{Eschenburg-Heintze}
Let $\L \subset \J$ be a Lagrangian subspace. If $I_\L = \R$ and $\tr (R) \geq 0$, then $\tr (R) = 0$ and $S^\L$ is identically $0$.
\end{proposition}

\begin{proof}  For the sake of completeness, let us briefly review the 
  idea of the proof extracted  from the  proof of   Theorem 1.7.1 of \cite{Gromoll-Walschap}.

\vspace{0.15\baselineskip}

Define $r: I_\L \subset \R \to \R$ by $r = \tr (R) + (\tr ({S^\L}^2) - \tr (S^\L)^2)$, as well as in Corollary \ref{Riccati-Equation-Real}. Since $\tr (R) \geq 0$ (by hypothesis) and in general $\tr ({S^\L}^2) \geq \tr (S^\L)^2$, both summands in the definition of $r$ are nonnegative and in particular $r \geq 0$.
We state that in this case $\tr (S^\L) = 0$. Suppose by contradiction that $\tr (S^\L) \neq 0$ or more specifically there exists $t_0 \in I_\L$ such that $\tr (S_{t_0}^{\L}) \neq 0$. Without loss of generality, assume that $t_0 = 0$. Then $t \longmapsto \tr (S_{t}^{\L})$ is a solution of the differential equation $x'(t) + x(t)^2 + r(t) = 0$ with a non null initial condition $x_0 = \tr (S_{t_0}^{\L})$ which implies that $\lim_{t \rightarrow {-\frac{1}{x_0}}^-} \tr (S_{t}^{\L}) = - \infty$, which is a contradiction. Finally Corollary \ref{Riccati-Equation-Real} implies that $r = 0$ and by the definition of $r$  we have that $\tr(R) = 0$ and $\tr ({S^\L}^2) = \tr (S^\L)^2$ which occurs if only if $S^\L = \frac{1}{n} \tr (S^\L) \text{\hspace{0.15ex}Id\hspace{0.15ex}}$. Then $S^\L = 0$.
\end{proof}
\vspace{0.25\baselineskip}

Following what was presented in the previous section, given a isotropic subspace $\I \subset \J$, we can associate a Riccati operator $S^{\L_\I}$ to any $\I$-transverse Lagrangian subspace $\L_\I$ (i.e. a lagrangian subspace of $\J_\I$). Furthermore, by Corollary \ref{transverse-Lagrangian-subspaces}, $S^{\L_\I}$ can be associated with a Lagrangian subspace $\L \subset \J$, such that $\L \supset \I$.

The Riccati operator $S^{\L_\I}$ will be called a $\I$-transverse Riccati operator (associated to $\L_\I$) and the Lemma \ref{Lagrangian-spaces-Riccati-operators} and Proposition \ref{Eschenburg-Heintze} holds to this type of Riccati operator either.

\vspace{0.25\baselineskip}

\subsection{Wilking's decomposition lemma} \label{Wilking-decomposition-lemma}

We are finally ready to enunciate and proof the Wilking's decomposition lemma. It is worth to mention that, as noted at the beginning of the Section \ref{Structures-associated-isotropic-subpaces-Jacobi-fields}, the central idea of the proof of this lemma is the choice of the following specific isotropic subspace
\begin{equation*}
\I = \mathrm{\hspace{0.1ex}span\hspace{0.25ex}}_{\mathbb{R}} \{J \in \L \;|\; J (t) = 0 \mathrm{\hspace{1ex} for \; some \hspace{1ex}} t \in \mathbb{R}\}
\end{equation*}
where $\L$ is a fixed Lagrangian subspace of a given Jacobi triple.

\begin{lemma}[Wilking's decomposition] \label{lemma-Wilking-decomposition}
Let $(E, D, R)$ be a Jacobi triple, such that the base of $E$ is $\mathbb{R}$ and $R$ is nonnegative. 
Then
\begin{equation*}
\L = \mathrm{\hspace{0.5ex}span\hspace{0.25ex}}_{\mathbb{R}} \{J \in \L \;|\; J(t) = 0 \mathrm{\hspace{1ex} for \; some \hspace{1ex}} t \in \mathbb{R}\} \oplus \{J \in \L \;|\; J \mathrm{ \hspace{1ex} is \; parallel}\},
\end{equation*}
for all $\L \in \Lambda (\J)$.
\end{lemma}

\begin{proof}
Define
\begin{equation} \label{definition-I}
\I := \mathrm{\hspace{0.1ex}span\hspace{0.25ex}}_{\mathbb{R}} \{J \in \L \;|\; J (t) = 0 \mathrm{\hspace{1ex} for \; some \hspace{1ex}} t \in \mathbb{R}\}.
\end{equation}
It is immediate that $\I \subset \J$ is isotropic subspace of Jacobi fields. Then denote $k = \dim (\I)$ and consider $V$ and $H$ the vertical and horizontal subbundles of $E$ induced by $\I$. Also denote by $\pi_H$ the horizontal projection whit respect to the decomposition $E = V \oplus H$ and by $\L_\I = \pi_H (\L) = \pi_H (\L / \I)$ the $\I$-tranversal Lagrangian subspace induced by $\I$ which is explicitly described by this
\begin{equation*}
\L_\I = \{J^H \in \J_\I \;|\; J \in \L\}
\end{equation*}
where $\J_\I$ is the space of $\I$-transverse Jacobi fields.

It is immediate that $\I \subset \J$ is isotropic subspace of Jacobi fields so that by Proposition \ref{transverse-Jacobi-fields} we have
\begin{equation*}
\L \approx \I \oplus \L / \I \approx \I \oplus \L_\I.
\end{equation*}
Then it suffices to prove that $\L_\I = \{J \in \L \;|\; J \mathrm{ \hspace{1ex} is \; parallel}\}$.
 Note that $\L_\I \supset \{J \in \L \;|\; J \mathrm{ \hspace{1ex} is \; parallel}\}.$ 
In fact for each $J \in \L$ parallel and $\tilde{J} \in \I$ we have that 
$  \scal{J}{\tilde{J}}'=\langle D J, \tilde{J}\rangle +  \langle J, D \tilde{J}\rangle = 2 \langle D J, \tilde{J}\rangle=0$ 
 and consequently $\scal{J}{\tilde{J}} = 0$ i.e. $J$ is horizontal.

In what follows we 
prove that $\L_\I \subset \{J \in \L \;|\; J \mathrm{ \hspace{1ex} is \; parallel}\}$ or equivalently that $J^H$ is $D$-parallel Jacobi field in $\L$ for all $J \in \L$. First we need to prove the next two claims.

\vspace{\baselineskip}
\textbf{Claim A.} $I_{\L_\I} = \R$ and the $\I$-transverse Riccati operator $S^{\L_\I}$ is identically $0$.
\vspace{0.25\baselineskip}

{\it Proof.\hspace{0.5ex}} First note that
\begin{equation} \label{equation-claim-a}
\I = \left\{ J \in \L \;|\; J(t) \in V_t \mathrm{\hspace{1ex} for \; some \hspace{1ex}} t \in \mathbb{R} \right\}.
\end{equation}
In fact, if $J(t) \in V_t$ for some $t \in \R$, there exists $J_1 \in \I$ and $J_2 \in \I^{0}_{t}$ (i.e. $J_2(t) = 0$) such that $J(t) = J_1(t) + D J_2 (t)$ (see item \ref{properties-isotropic-c} of Lemma \ref{properties-isotropic}). 
By mutipling both sides of the equation by $D J_2 (t)$ and using the fact that $\omega(J, J_2)=\omega(J_1, J_2)=0$ we can infer that
$\| D J_2 (t) \|^2 = 0$. 
We conclude that $J - J_1$ is a Jacobi field in $\L$ such that $(J - J_1) (t) = 0$ and by the definition of $\I$ (see Eq. \eqref{definition-I}) $J - J_1 \in \I$ or equivalently $J \in \I$.  The other inclusion follows from the fact that each $J\in \I$ is vertical. 

Now take a $\L$-regular instant $t_0 \in \R$ and $J_1, \cdots, J_{n-k} \in \L$ such that $\{J^{H}_{1} (t_0), \cdots, J^{H}_{n-k} (t_0)\} \subset H_{t_0}$ is a base. Then $\{J^{H}_{1}, \cdots, J^{H}_{n-k}\}$ is a frame of $H$. In fact for each $t \in \R$ if $\sum \lambda_i J^{H}_{i} (t) = 0$ then $\sum \lambda_i J_i \in \I$ (by Eq. \eqref{equation-claim-a}) which implies that $\sum \lambda_i J^{H}_{i} (t_0) \in H_{t_0} \cap V_{t_0} = \{0\}$ and follows that $\lambda_i = 0$.

Finally, since $H$ has a global frame of the form $\{J^{H}_{1}, \cdots, J^{H}_{n-k}\}$, we conclude that $\mathrm{\hspace{0.1ex}span\hspace{0.25ex}}_{\mathbb{R}} \{J^{H}_{1}, \cdots, J^{H}_{n-k}\} = \L_\I$ and consequently $\L_\I (t) = H_t$ which by definition of $L_\I$-regular instants means that $I_{\L_\I} = \R$  (see eq. \ref{regular-instants} and remember that the total space for $\I$-transverse Jacobi fields is $H$).

Additionally, since $R$ is nonnegative, Proposition \ref{Eschenburg-Heintze} implies that $S^{\L_\I} \equiv 0$.

\vspace{\baselineskip}
\textbf{Claim B.} $R_H, \mathbb{A} = 0$.
\vspace{0.25\baselineskip}

{\it Proof.\hspace{0.5ex}} Since $S^{\L_\I} = 0$ (see {\bf Claim A}), it follows by the transverse version of Riccati equation, i.e. 
$(S^{\L_\I})' + (S^{\L_\I})^2 + (R_H - 3 \mathbb{A}^{2} |_{H}) =0 $ (see Eq \ref{Riccati-equation}), that $R_H = 3 \mathbb{A}^{2} |_{H}.$ 
This equation together with Eq.\eqref{eq-propriedades-tensor-A-Aadjunto} imply that $\forall X$
\begin{align*}
\langle R_H X, X\rangle & =  \langle 3 \mathbb{A}^{2} |_{H}X, X  \rangle \\ 
 & =  -3\langle \mathbb{A}_{H}^{\ast} \mathbb{A}_H X,X\rangle\\
&=  -3 \langle  \mathbb{A}_H X,\mathbb{A}_H X\rangle
\end{align*}
Therefore, since by hypothesis $R_H$ is nonnegative, we infer that $R_H = 0$  and hence $\mathbb{A}= 0$.


\vspace{\baselineskip}

Now we are going to prove that $J^H$ is $D$-parallel for all $J \in \L$. Indeed $J^H$ is $D_H$-parallel since $D_H J^H = S^{\L_\I} J^H = 0$ (see Claim A and item \ref{Lagrangian-spaces-Riccati-operators-a} of Proposition \ref{Lagrangian-spaces-Riccati-operators}) which means that $(D J^H)^H = 0$. The nullity of the vertical component of $D J^H$ follows from the fact that 
$\scal{(D J^H)^V}{\tilde{J}} = \scal{\mathbb{A}J^H}{\tilde{J}} = 0$ for all $\tilde{J} \in \I$.

Now we proceed with the prove that $J^H$ is a Jacobi field. Since $J^H$ is $D$-parallel it suffices to prove that $R J^H = 0$. By $R$ self-adjointness and $J^H$ $D$-parallelism, for each $\tilde{J} \in \I$, follows that

\begin{align*}
\Scal{(R J^H)^V}{\tilde{J}} &= \Scal{R J^H}{\tilde{J}} \\
&= \Scal{J^H}{R \tilde{J}} \\
&= \Scal{J^H}{- D^2 \tilde{J}} \\
&= - \Scal{J^H}{D \tilde{J}}' \\
&= - \Scal{J^H}{\tilde{J}}'' \\
&= 0.
\end{align*}

Then by {\bf Claim B} we conclude that $R J^H = (R J^H)^V + R_H J^H = 0$.

Finally we finish with the proof that $J^H \in \L$. Indeed it is immediately from $J^H$ $D$-parallelism that $\omega (J^H, \tilde{J}) = - \scal{J^H}{\tilde{J}}' = - \scal{J^H}{\tilde{J}^H}' = 0$ for all $\tilde{J} \in \L$ which implies that $J^H \in \L^\omega = \L$.
\end{proof}

As a direct  consequence  of  Wilking's decomposition lemma, 
we can infer  Proposition \ref{lemma-Jacobi-Decomposition}.

\vspace{\baselineskip}

\bibliographystyle{amsplain}

\end{document}